\documentclass[reqno]{amsart}  
\theoremstyle{plain}
\newtheorem{theorem}{Theorem}[section]
\newtheorem{lemma}[theorem]{Lemma}
\newtheorem{remark}[theorem]{Remark}
\newtheorem{proposition}[theorem]{Proposition}

\newtheorem{example}[theorem]{Example}
\numberwithin{equation}{section}
\allowdisplaybreaks[1]

\theoremstyle{definition}

\newtheorem{definition}[theorem]{Definition}

\theoremstyle{remark}

\newcommand{\bQ}{{Z_{\rm row}}}

\newcommand{\bU}{{\mathbf U}}

\newcommand{\bz}{{\mathbf z}}
\newcommand{\bzeta}{{\boldsymbol \zeta}}

\newcommand{\bT}{{\mathbf T}}
\newcommand{\lam}{\lambda}

\newcommand{\bn}{{\mathbf n}}

\newcommand{\cD}{{\mathcal D}}
\newcommand{\cF}{{\mathcal F}}
\newcommand{\cH}{{\mathcal H}}
\newcommand{\cL}{{\mathcal L}}
\newcommand{\cM}{{\mathcal M}}
\newcommand{\cR}{{\mathcal R}}
\newcommand{\cK}{{\mathcal K}}
\newcommand{\cfm}{\mbox{\bf{c.f.m.}}}
\newcommand{\dcfm}{\mbox{\bf{d.c.f.m.}}}
\newcommand{\tcfm}{\mbox{\bf{t.c.f.m.}}}
\newcommand{\cS}{{\mathcal S}}

\newcommand{\cU}{{\mathcal U}}
\newcommand{\cX}{{\mathcal X}}
\newcommand{\cY}{{\mathcal Y}}

\newcommand{\D}{{\mathbb D}}

\newcommand{\C}{{\mathbb C}}
\newcommand{\B}{{\mathbb B}}

\newcommand{\tA}{\widetilde{A}}
\newcommand{\tB}{\widetilde{B}}
\newcommand{\tC}{\widetilde{C}}
\newcommand{\tD}{\widetilde{D}}

\newcommand{\sbm}[1]{\left[\begin{smallmatrix} #1
		\end{smallmatrix}\right]}

\begin{document}

\title[Transfer-function realization]{Canonical transfer-function
realization for Schur multipliers on the Drury-Arveson space and 
models for commuting row contractions}
\author[J. A. Ball]{Joseph A. Ball}
\address{Department of Mathematics,
Virginia Tech,
Blacksburg, VA 24061-0123, USA}
\email{ball@math.vt.edu}
\author[V. Bolotnikov]{Vladimir Bolotnikov}
\address{Department of Mathematics,
The College of William and Mary,
Williamsburg VA 23187-8795, USA}
\email{vladi@math.wm.edu}

\begin{abstract}
We develop a $d$-variable analog of the two-component de 
Bran\-ges-Rovnyak reproducing kernel Hilbert space associated with a 
Schur-class function on the unit disk.  In this generalization, the 
unit disk is replaced by the unit ball in $d$-dimensional complex 
Euclidean space, and the Schur class becomes the class of contractive multipliers 
on the Drury-Arveson space over the ball.  We also develop some 
results on a model theory for commutative row contractions which are 
not necessarily completely noncoisometric (the case 
considered in earlier work of Bhattacharyya, Eschmeier and Sarkar).
\end{abstract}

\subjclass{47A13, 47A45, 47A48}
\keywords{Operator-valued Schur-class functions, Agler decomposition,
unitary realization, operator model theory}

\maketitle

\section{Introduction}  \label{S:Intro}
\setcounter{equation}{0}
For $\cU$ and $\cY$ two Hilbert spaces we let $\cL(\cU,\cY)$
be the space of bounded linear operators mapping
$\cU$ into $\cY$, abbreviated to $\cL(\cU)$ in case $\cU=\cY$.
The operator-valued version of the classical Schur class ${\mathcal
S}(\cU, \cY)$ is defined to be the set of all holomorphic, contractive
$\cL(\cU, \cY)$-valued functions on the unit disk ${\mathbb D}$. With any such
function
$S: \D\to\cL(\cU,\cY)$, one can associate the following three
operator-valued kernels
\begin{equation}
K_{S}(z, \zeta) = \frac{ I_{\cY} - S(z) S(\zeta)^{*}}{1
-z\overline{\zeta}},\quad
\widetilde{K}_S(z,\zeta)=\frac{I_{\cU}-S(z)^*S(\zeta)}{1-\overline{z}\zeta},
\label{1.1}
\end{equation}
\begin{equation}
\label{1.2}
\widehat K(z,\zeta) =
\begin{bmatrix} K_{S}(z, \zeta)  & 
{\displaystyle\frac{ S(z) -S(\zeta)}{ z -\zeta}}\\
{\displaystyle\frac{ S(z)^{*} - 
S(\zeta)^{*}}{ \overline{z} - \overline{\zeta}}} &
\widetilde{K}_{S}(z, \zeta) \end{bmatrix}.
\end{equation}  
The Schur class ${\mathcal S}(\cU, \cY)$ can be characterized as the
set 
of all $\cL(\cU, \cY)$-valued function on $\D$ for which any (and 
therefore every) of the above three kernels is positive on
$\D\times\D$.
Furthermore, for every function $S\in\cS(\cU,\cY)$ there exists  an 
auxiliary Hilbert space $\cX$ and a unitary connecting operator (or 
colligation) 
\begin{equation}
\label{1.3}  
\bU =\begin{bmatrix}A & B \\ C & D\end{bmatrix}\colon
\begin{bmatrix}\cX \\ \cU\end{bmatrix}\to \begin{bmatrix}\cX \\ 
\cY\end{bmatrix}
\end{equation}
so that $S(z)$ can be expressed as
\begin{equation}  
\label{1.4}
S(z) = D +z C (I -z A)^{-1} B \quad\mbox{for all}\quad z\in\D.
\end{equation}
On the other hand, if $\bU$ of the form \eqref{1.3} is a contraction, 
then the function $S$ of the form \eqref{1.4} belongs to
$\cS(\cU,\cY)$.
The formula \eqref{1.4} is called a {\em realization} of the
function $S$ which in turn, is called the {\em characteristic
function} of
the colligation \eqref{1.3}. The realization is called unitary, 
isometric, coisometric or contractive if the connecting operator
$\bU$ is 
respectively, unitary, isometric, coisometric or contractive. It is
seen from \eqref{1.4} that for any realization $\bU$ of $S$, the entry $D$
is uniquely determined and equals $S(0)$. As was shown 
in \cite{dbr1}, \cite{dbr2}, \cite{dBS} for unitary, isometric or 
coisometric realizations, the state space $\cX$ and the operators
$A$, $B$  and $C$ can be chosen in a certain canonical way (specific for each
type) and these realizations are unique up to unitary equivalence under
certain minimality conditions which we now recall. With a colligation
\eqref{1.3} we associate the observability subspace $\cH^{\mathcal
O}_{C,A}$ and the controllability subspace $\cH^{\mathcal C}_{A,B}$ by
\begin{equation}
\cH^{\mathcal O}_{C,A}:=\bigvee_{n\ge 0} {\rm Ran}A^{*n}C^*, \qquad
\cH^{\mathcal C}_{A,B}:=\bigvee_{n\ge 0} {\rm Ran}A^{n}B,
\label{1.5}
\end{equation}
where $\bigvee$ denotes the closed linear span. 
\begin{definition}
The colligation $\bU=\sbm{A & B \\ C & D}\colon
\sbm{\cX \\ \cU}\to \sbm{\cX \\ \cY}$ is called {\em observable}, 
{\em controllable} or {\em closely connected} if respectively,
$$
\cH^{\mathcal O}_{C,A}=\cX,\quad \cH^{\mathcal 
C}_{A,B}=\cX\quad\mbox{or}\quad \cH^{\mathcal
O}_{C,A}\bigvee\cH^{\mathcal 
C}_{A,B}=\cX.
$$ 
Furthermore, $\bU$ is called {\em unitarily equivalent} to a
colligation
\begin{equation}
\widetilde {\mathbf U}
= \begin{bmatrix}\tA& \tB \\ \tC & D\end{bmatrix}
\colon\; \begin{bmatrix}\widetilde{\cX} \\  \cU\end{bmatrix}\to 
\begin{bmatrix}\widetilde{\cX}\\ 
\cY\end{bmatrix}
\label{1.6}
\end{equation}
if there exists a unitary operator $U\colon \cX \to\widetilde{\cX}$
such that
\begin{equation}
UA=\tA U,\quad UB=\tB\quad\mbox{and}\quad C=\tC U.
\label{1.7}
\end{equation}
\label{D:1.2}
\end{definition}
We remark that some authors attribute the notions of observability and
controllability to the pairs $(C,A)$ and $(A,B)$ rather than to the
whole colligation $\bU$, and call the colligation with an  observable
output pair $(C,A)$ and/or with a controllable input pair  $(A,B)$ 
respectively  {\em closely outer-connected} and/or {\em  closely 
inner-connected}.

\smallskip

Recall that given an $S\in\cS(\cU,\cY)$, the three associated kernels 
in \eqref{1.1} and \eqref{1.2} are positive and give rise to the 
respective  reproducing kernel Hilbert spaces $\cH(K_{S})$,
$\cH(\widetilde{K}_{S})$ and $\cH(\widehat{K}_{S})$ (called
{\em de Branges-Rovnyak reproducing kernel Hilbert spaces}).  Observe
that 
the kernel $K_S(z,\zeta)$ is analytic in $z,\overline{\zeta}$ and
therefore, all functions in the associated space $\cH(K_S)$ are
analytic
on $\D$. The kernel $\widetilde{K}_S$ is analytic in $\overline{z}$
and
$\zeta$ and the associated space $\cH(\widetilde{K}_{S})$ consists of
conjugate-analytic functions. Similarly, the elements of
$\cH(\widehat{K}_{S})$ are the functions of the form $f=\sbm{f_+
\\ f_-}$ where $f_+$ is analytic and $f_-$ is
conjugate-analytic. The following theorem summarizes realization
results 
from \cite{dbr1}, \cite{dbr2}, \cite{dBS}. 

\begin{theorem}
\label{T:sum}
Let $S\in\cS(\cU,\cY)$ and let $D:=S(0)$.
\begin{enumerate} 
\item Let $\cX=\cH(K_S)$ and let
\begin{equation}
A \colon f(z) \mapsto\frac{f(z) - f(0)}{z},\quad
B \colon u\mapsto \frac{S(z) - S(0)}{z} u, \quad
C \colon f(z) \mapsto f(0),
\label{1.8}
\end{equation}
Then ${\bf U}$ \eqref{1.3} is an observable coisometric colligation 
with its characteristic function equal to $S$. Any observable 
coisometric colligation $\widetilde {\mathbf U}$ \eqref{1.7}
with its characteristic function equal to $S$ is unitarily equivalent to 
$\bU$.
\item Let $\cX=\cH(\widetilde{K}_S)$ and let $B \colon u \mapsto 
(I_\cU-S(z)^*S(0))u$,
\begin{equation}
A^* \colon f(z) \mapsto\frac{f(z) - f(0)}{\overline{z}},\quad 
C^* \colon y  \mapsto {\displaystyle\frac{S(z)^* -
S(0)^*}{\overline{z}}} y.
\label{1.9}
\end{equation}
Then ${\bf U}$ \eqref{1.3} is a controllable isometric colligation
with its characteristic function equal to $S$. Any controllable
isometric colligation $\widetilde {\mathbf U}$ \eqref{1.7}
with its characteristic function equal to $S$ is unitarily equivalent to
$\bU$.
\item Let $\cX=\cH(\widehat{K}_S)$ and let
\begin{equation}   \label{1.9'}
\begin{array}{ll}
A \colon \begin{bmatrix} f(z) \\ g(z) \end{bmatrix}\mapsto
\begin{bmatrix} [f(z) - f(0)]/z \\ \overline{z}g(z) - S(z)^* f(0)
\end{bmatrix}, & \quad B \colon u \mapsto \begin{bmatrix}
\frac{S(z) - S(0)}{z} u \\ (I - S(z)^{*} S(0))u \end{bmatrix},  \\
C \colon \begin{bmatrix} f(z) \\ g(z) \end{bmatrix} \mapsto f(0).
  \end{array}
\end{equation}
Then ${\bf U}$ \eqref{1.3} is a closely connected unitary 
colligation with its characteristic function equal to $S$. Any 
closely connected unitary colligation with its characteristic 
function equal to $S$ is unitarily equivalent to $\bU$.
\end{enumerate}
\end{theorem}

We mention that such de Branges-Rovnyak reproducing kernel Hilbert 
spaces can be used as canonical functional 
model Hilbert spaces for contraction operators of various classes 
(namely, completely noncoisometric, completely nonisometric, and 
completely nonunitary)---see \cite{NF, BK}.

\smallskip

The objective of this paper is to extend these realization results 
to the following multivariable setting. We denote by 
$\B^d=\left\{z=(z_1,\dots,z_d)\in\C^d \colon
\langle z, z\rangle<1\right\}$ the unit ball of the Euclidean space
$\C^d$ with the standard inner product $\langle z,  
\zeta\rangle=\sum_{j=1}^d z_j\bar{\zeta}_j$. The kernel 
$k_d(z,\zeta)=\frac{1}{1-\langle z,\zeta\rangle}$ is positive on 
${\mathbb B}\times {\mathbb B}$ and we denote by $\cH(k_d)$ the
associated 
reproducing kernel Hilbert space (the Drury-Arveson space) which for 
$d=1$ is the usual Hardy space $H^2$ of the unit disk. For a Hilbert
space 
$\cY$, we use notation $\cH_\cY(k_d)$ for the Drury-Arveson space of
$\cY$-valued functions which can be characterized in terms of power
series 
as follows:
\begin{equation}
\cH_{\cY}(k_d)=\left\{f(z)=\sum_{\bn \in{\mathbb Z}^d_{+}}f_{n}
z^n:\|f\|^{2}=\sum_{n \in {\mathbb Z}^d_{+}}
\frac{n!}{|n|!}\|f_{n}\|_{\cY}^2<\infty\right\}.
\label{char}
\end{equation}
Here and in what follows, we use standard multivariable notations: for
multi-integers $n =(n_{1},\ldots,n_{d})\in{\mathbb Z}_+^d$ and points
$z=(z_1,\ldots,z_d)\in\C^d$ we set
\begin{equation}
|n| = n_{1}+n_{2}+\ldots +n_{d},\quad
n!  = n_{1}!n_{2}!\ldots n_{d}!, \quad
z^\bn = z_{1}^{n_{1}}z_{2}^{n_{2}}\ldots
z_{d}^{n_{d}}.
\label{mnot}
\end{equation}
Given two Hilbert spaces $\cU$ and $\cY$, we denote by
$\cS_d(\cU,\cY)$ the 
class of $\cL(\cU,\cY)$-valued functions 
$S$ on $\B^d$ such that the multiplication operator $M_S: \; f\mapsto 
S\cdot f$ defines a contraction from $\cH_{\cU}(k_d)$ into 
$\cH_{\cY}(k_d)$ or equivalently, such that the kernel
\begin{equation}
\label{1.10}
K_{S}(z, \zeta) = \frac{ I_{\cY} - S(z) S(\zeta)^{*}}
{1 - \langle z, \zeta \rangle}
\end{equation}
is  positive on ${\mathbb B}\times {\mathbb B}$. It is readily seen
that 
the class $S_1(\cU,\cY)$ is the Schur class introduced above. In
general, 
it follows from positivity of $K_S$ that $S$ is holomorphic and takes 
contractive values on $\B^d$. However, for $d>1$ there are analytic 
contractive-valued functions on $\B^d$ not in $\cS_d$. The class 
$\cS_d(\cU,\cY)$ can be characterized in 
various ways similarly to the one-variable situation. Here we recall
the  one from  \cite{BTV} given in terms of norm-constrained realizations.

\smallskip

In what follows we use notation $Z_{\rm row}(z)=\begin{bmatrix}z_1 &
\ldots &
z_d\end{bmatrix}$ and for a Hilbert space $\cX$ we let 
\begin{equation} \label{1.12a}
Z_{\cX}(z):=Z_{\rm row}(z)\otimes I_{\cX}=
\begin{bmatrix}z_1 I_\cX & \ldots & z_d I_\cX \end{bmatrix}.  
\end{equation}
\begin{theorem}   \label{T:1.1}
If a function $S \colon {\mathbb B}^{d} \to \cL(\cU, \cY)$ belongs to 
$\cS_{d}(\cU,\cY)$, then there is an auxiliary Hilbert space $\cX$
and a unitary connecting operator (or colligation)
\begin{equation}
\label{1.11}
\bU = \begin{bmatrix} A & B \\ C & D \end{bmatrix} =
\begin{bmatrix}  A_{1} & B_{1} \\ \vdots & \vdots \\ A_{d} & B_{d}
    \\ C & D \end{bmatrix} \colon
\begin{bmatrix} \cX \\ \cU \end{bmatrix} \to \begin{bmatrix} \cX^{d}\\
\cY \end{bmatrix}
\end{equation}
so that $S(z)$ can be realized as
\begin{align}  \label{1.12}   
S(z) &=D + C
(I-z_1A_1-\cdots-z_dA_d)^{-1}(z_1B_1+\ldots+z_dB_d)\notag\\
&=D + C (I_{\cX} - Z_{\cX}(z) A)^{-1}Z_{\cX}(z) B\qquad (z\in\B^d).
\end{align}
Conversely, if $\bU$ of the form \eqref{1.11} is a 
contraction, then
the function $S$ of the form \eqref{1.12} belongs to
$\cS_{d}(\cU,\cY)$.
\end{theorem}
As in the univariate case, $S$ of the form \eqref{1.12} will be 
referred to as the {\em characteristic function} of the colligation 
\eqref{1.11}.
The main goal of the present paper is to establish the analog of 
Theorem \ref{T:sum} for the
present multivariable setting, that is to obtain coisometric,
isometric 
and unitary functional-model realizations of a given $S\in\cS_d(\cU,\cY)$ 
in a certain canonical way and to show that these types of
realizations are 
unique up to unitary equivalence under suitable minimality
conditions.  As an application of our functional model spaces for 
this multivariable ball setting, we show how the model theory of 
Bhattacharyya-Eschmeier-Sarkar \cite{BES, BES2} for commutative 
row-contractive operator $d$-tuples can be extended beyond the 
completely noncoisometric case, and we relate these results with 
those of the first author and Vinnikov \cite{Cuntz-scat} established 
for general (possibly noncommutative) completely nonunitary 
row-contractive operator $d$-tuples.

\smallskip

We now introduce the minimality conditions which will play a key role 
in the sequel.
We denote by ${\mathcal I}_{i}: \, \cX \to \cX^{d}$ the  inclusion 
map of the space $\cX$ into the $i$-th slot in the direct-sum space $\cX^{d} = 
\bigoplus_{k=1}^{d} \cX$; the adjoint then is the orthogonal 
projection of $\cX^{d}$ down to the $i$-th coordinate:
\begin{equation}  \label{1.13}
{\mathcal I}_{i} \colon  x_i \mapsto \begin{bmatrix} 0 \\ \vdots \\
x_i \\ \vdots \\  0 \end{bmatrix}\quad \text{and}\quad
{\mathcal I}_{i}^{*} \colon \begin{bmatrix}x_1 \\ \vdots \\ x_i \\ \vdots
\\ x_d \end{bmatrix} \mapsto  x_i.
\end{equation}
With a structured colligation \eqref{1.11}
we associate the observability subspace $\cH^{\mathcal
O}_{C,A}$ and the controllability subspace $\cH^{\mathcal C}_{A,B}$
as 
follows:
\begin{align}
\cH^{\mathcal
O}_{C,A}&:=\bigvee\left\{(I_{\cX}-A^*Z_{\cX}(z)^*)^{-1}C^*y: 
\; z\in \B^d, \; y\in\cY\right\},\label{1.14}\\
\cH^{\mathcal C}_{A,B}&:=\bigvee_{j=1}^d\left\{{\mathcal I}_{j}^{*}
(I_{\cX^d}-AZ_{\cX}(z))^{-1}Bu: \; z\in \B^d, \; u\in\cU\right\},
\label{1.15}
\end{align}
where $Z_{\cX}$ and ${\mathcal I}_{j}^{*}$ are given in \eqref{1.12a} and 
\eqref{1.13}. Observe, that in case $d=1$ these definitions are
equivalent to 
those in \eqref{1.5}. Similarity becomes more transparent if one
writes 
definitions \eqref{1.14}, \eqref{1.15} in terms of powers of the 
state space operators $A_1,\ldots,A_d$; however, at this point we try
to avoid power 
notation which requires some more explanations and notation in case
the state 
space operators do not commute. With the spaces \eqref{1.14},
\eqref{1.15} in 
hand, the multivariable extension of Definition \ref{D:1.2} is now 
immediate.
\begin{definition}
The structured colligation $\bU=\sbm{A & B \\ C & D}\colon
\sbm{\cX \\ \cU}\to \sbm{\cX^d \\ \cY}$ as in \eqref{1.11} is called
{\em 
observable}, {\em controllable} or {\em closely connected} if
respectively,
$$
\cH^{\mathcal O}_{C,A}=\cX,\quad \cH^{\mathcal
C}_{A,B}=\cX\quad\mbox{or}\quad \cH^{\mathcal
O}_{C,A}\bigvee\cH^{\mathcal
C}_{A,B}=\cX.
$$
Furthermore, $\bU$ is called {\em unitarily equivalent} to a
colligation
\begin{equation}
\widetilde {\mathbf U}
= \begin{bmatrix}\tA& \tB \\ \tC & D\end{bmatrix}
\colon\; \begin{bmatrix}\widetilde{\cX} \\  \cU\end{bmatrix}\to
\begin{bmatrix}\widetilde{\cX}^d\\
\cY\end{bmatrix}
\label{1.16}
\end{equation}
if there exists a unitary operator $U\colon \cX \to\widetilde{\cX}$
such that
\begin{equation}
(\oplus_{i=1}^{d} U)A=\tA U,\quad (\oplus_{i=1}^{d} 
U)B=\tB\quad\mbox{and}\quad C=\tC U.
\label{1.17}
\end{equation}
\label{D:1.4}
\end{definition}
It is readily seen that equalities \eqref{1.17} is what we need to
guarantee 
(as in the univariate case) that the  characteristic  functions of
${\bf U}$ 
and of $\widetilde {\mathbf U}$ are equal.

\smallskip

As was pointed out on many occasions, a more useful analog of 
coisometric (isometric, unitary)  realizations appearing in the
classical
univariate case is not that the whole connecting operator ${\bf U}$
be coisometric (isometric, unitary), but rather that ${\bf U}$ and/or
${\bf
U}^*$ be contractive and isometric on certain canonical subspaces
closely related to the subspaces \eqref{1.14} and \eqref{1.15}. 
\begin{definition}
A contractive colligation  $\bU$ of the form \eqref{1.11} is called
\begin{enumerate}
\item {\em weakly isometric} if $\bU$ is isometric on the   
subspace $\widetilde{\mathcal D}_{A,B}\oplus\cU$ where
$$
\widetilde{\mathcal D}_{A,B}:=
\bigvee_{\zeta\in\B^d, \, u\in\cU}Z_{\cX}(\zeta)(I -
AZ_{\cX}(\zeta))^{-1}Bu\subset \cX;
$$
\item {\em weakly coisometric} if the
adjoint $\bU^*: \, \cX^d\oplus\cY\to \cX\oplus\cU$ is
isometric on the subspace ${\mathcal D}_{C,A}\oplus\cY$ where
\begin{equation}
\cD_{C,A}:= \bigvee_{\zeta \in \B^d,y\in\cY}Z_{\cX}(\zeta)^{*} (I -
A^{*}
Z_{\cX}(\zeta)^{*})^{-1} C^{*} y\subset\cX^{d};
\label{1.18}
\end{equation}
\item {\em weakly unitary} if it is weakly isometric and weakly
coisometric.
\end{enumerate}
\label{D:1.5}
\end{definition}
We remark that the above weak notions do not appear in the
single-variable 
case for a simple reason that if the pair $(C,A)$ is observable than 
$\cD_{C,A}=\cX$ so that a weakly coisometric colligation is
automatically 
coisometric and similarly, if the pair $(A,B)$ is controllable, then 
$\widetilde{\mathcal D}_{A,B}=\cX$ so that a weakly isometric
colligation is 
automatically isometric.

\smallskip

As was shown in \cite{BBF2a}, a function $S\in\cS_d(\cU,\cY)$ may not
admit
an observable coisometric realization. In contrast, observable 
weakly-coisometric realizations always exist and up to unitary
equivalence,
all these realizations are canonical functional model (\textbf{c.f.m.})
realizations (see Definition \ref{D:2.3} below) with the state space 
equal to the de Branges-Rovnyak space $\cH(K_S)$ with reproducing
kernel 
\eqref{1.10},
with  the output operator $C$ equal to evaluation at zero on
$\cH(K_{S})$,
and with operators $A$ and $B$ whose adjoints are uniquely
determined on the subspace $\cD_{C,A}\subset\cX^{d}$ given in
\eqref{1.18}.
Realizations of this type were studied in \cite{BBF2a}, \cite{BBF2b}, 
\cite{bb} and will be briefly reviewed in Section \ref{S}.  Section 
\ref{T} is a brief sketch of the dual canonical functional model 
(\textbf{d.c.f.m.}) colligations which provide a canonical functional 
model for controllable weakly isometric realizations of $S$ (the 
analog of part (2) of Theorem \ref{T:sum}); this section is kept 
quite short as the proofs of the results can be seen as special cases 
of the more general manipulations carried out in Section \ref{S:unit}.
Section \ref{S:unit} is the core of the paper where the theory of the 
two-component canonical functional model (\textbf{t.c.f.m.}) 
colligations (the analog of part (3) of Theorem \ref{T:sum}) is 
carried out; these form the precise class of canonical models 
which provide weakly unitary  closely connected realizations for the 
contractive Drury-Arveson-space multiplier $S$.  
Section \ref{S:CharF} gives the application to the model theory for 
row-contractive operator $d$-tuples, i.e., 
how these \textbf{t.c.f.m.} colligations 
can be used to extend at least partially the role of the de 
Branges-Rovnyak two-component spaces $\cH(\widehat K_{S})$ as the 
model space for completely nonunitary contractions to the setting of 
commutative row-contractive operator $d$-tuples $\bT =(T_{1}, \dots, 
T_{d})$ where each $T_{k}$ ($k=1, \dots, d$) is an operator on a 
fixed Hilbert space $\cX$ and the block row-matrix $T = \begin{bmatrix} T_{1} & \cdots & 
T_{d} \end{bmatrix}$ is a contraction operator from $\cX^{d}$ to $\cX$.
Here we also give a simple example (specifically, a 
point on the boundary of the unit ball in ${\mathbb C}^{2}$ viewed as 
a 2-tuple of operators on ${\mathbb C}$) for which the completely 
noncoisometric version of the model theory gives no information but 
for which our added invariant gives complete information.  The final 
Section \ref{S:NAD} sketches on the \textbf{t.c.f.m.}~approach to 
operator-model theory leads 
to more definitive results for unitary classification of not 
necessarily commutative row-contractive operator-tuples; here we draw on 
results from  \cite{Cuntz-scat} and \cite{BV-formal}.

\smallskip

We close the Introduction with a short discussion of the {\em Gleason 
problem} to give the reader some orientation for the multivariable 
formalisms to follow.   The reader will note that the {\em 
difference-quotient transformation} 
$$
 f(z) \mapsto [f(z) - f(0)]/z
$$
(where $f$ is a function which is holomorphic at $0$)
plays a key role in the definition of the model operators $A,B,C^{*}$ 
in Theorem \ref{T:sum}.  For some time now it has been recognized 
that the multivariable analog of the difference-quotient 
transformation is any solution of the so-called {\em Gleason problem} 
for a space of holomorphic functions ${\mathcal H}$ (see 
\cite{gleason, AD1, AD2, AKap}).  Given a space ${\mathcal H}$ of 
holomorphic functions $h$ which are holomorphic in a neighborhood of 
the origin in $d$-dimensional complex Euclidean space ${\mathbb 
C}^{d}$, we say that the operators $R_{1}, \dots, R_{d}$ mapping 
${\mathcal H}$ into itself {\em solve the Gleason problem for 
${\mathcal H}$} if every function $h \in {\mathcal H}$ has a 
decomposition (not necessarily unique) of the form 
$$
h(z) = h(0) + \sum_{k=1}^{d} z_{k} [R_{k}h](z).
$$
We shall see that more structured variations on this idea appear in the 
definition of the model operators for the various canonical 
functional-model spaces over the ball ${\mathbb B}^{d}$ to appear in the sequel. 

\section{Weakly coisometric canonical realizations}
\label{S}

For any $S\in\cS_d(\cU,\cY)$, the associated kernel $K_S$ \eqref{1.10}
is positive on $\B^d\times\B^d$ so we can associate with $S$ the 
de Branges-Rovnyak reproducing kernel Hilbert space $\cH(K_S)$.
In parallel to the  univariate case, $\cH(K_S)$ is the
state space of certain canonical functional-model realization for $S$.
\begin{definition}
We say that the contractive operator-block matrix
\begin{equation}
\bU = \begin{bmatrix} A & B \\ C & D\end{bmatrix}\colon \;
\begin{bmatrix}\cH(K_S)  \\ \cU\end{bmatrix}\to
\begin{bmatrix}\cH(K_S)^{d}  \\ \cY\end{bmatrix}
\label{2.8}
\end{equation}
is a {\em canonical functional-model
({\rm abbreviated to $\cfm$ in what follows}) colligation} for the
given 
function $S\in{\mathcal {S}}_d(\cU, \cY)$ if
\begin{enumerate}
\item The operator $A = \left[ \begin{smallmatrix} A_{1} \\ \vdots \\ 
A_{d} \end{smallmatrix} \right]$ solves the Gleason problem for
$\cH(K_S)$, i.e., 
\begin{equation}
f(z)-f(0)=\sum_{j=1}^d z_j (A_jf)(z)\quad\mbox{for all}\quad
f\in\cH(K_S).
\label{2.6}
\end{equation}
\item The operator $B = \left[ \begin{smallmatrix} B_{1} \\ \vdots \\ 
B_{d} \end{smallmatrix} \right]$ solves the Gleason problem for $S$:
\begin{equation}
 S(z) u - S(0) u = \sum_{j=1}^{d}z_j (B_ju)(z)\quad\mbox{for
all}\quad u \in \cU.
\label{2.7}
\end{equation}
\item The operators $C: \, \cH(K_S)\to \cY$ and $D: \, \cU\to
\cY$ are given by
\begin{equation}
C \colon f\mapsto f(0), \quad  D \colon u \mapsto S(0) u.
\label{2.7a}
\end{equation}
\end{enumerate}
\label{D:2.3}   
\end{definition}
We next rearrange equality \eqref{1.10} as follows
\begin{equation}
\sum_{j=1}^{d}z_j\bar{\zeta}_jK_S(z,\zeta)+I_{\cY}=
K_S(z,\zeta)+S(z)S(\zeta)^*
\label{2.4}   
\end{equation}
and write \eqref{2.4}  in the inner product form as
\begin{align*}
&\left\langle \left[\begin{array}{c}
{\bQ}(\zeta)^*\otimes K_S(\cdot,\zeta)y \\ y\end{array}\right], \;
\left[\begin{array}{c} {\bQ}(z)^*\otimes K_S(\cdot,z)y' \\
y'\end{array}\right]\right\rangle_{\cH(K_S)^d\oplus\cY}
\notag\\
&\qquad=\left\langle \left[\begin{array}{c}
K_S(\cdot,\zeta)y \\ S(\zeta)^*y\end{array}\right], \;
\left[\begin{array}{c}K_S(\cdot,z)y' \\
S(z)^*y'\end{array}\right]\right\rangle_{\cH(K_S)\oplus\cU}.
\end{align*}
It now follows that the map
\begin{equation}
V\colon
\;\left[\begin{array}{c}Z_{\cH(K_{S})}(\zeta)^*K_S(\cdot,\zeta)y
\\ y\end{array}\right] \to
\left[\begin{array}{c}K_S(\cdot,\zeta)y
 \\ S(\zeta)^*y\end{array}\right]
\label{2.5}
\end{equation}
extends by linearity and continuity to an isometry with initial space
$$
{\mathcal D}_V={\displaystyle\bigvee_{\zeta\in\B^d, \, y\in 
\cY}\begin{bmatrix}Z_{\cH(K_{S})}(\zeta)^*K_S(\cdot,\zeta)y\\ 
y\end{bmatrix}}.
$$ 
The following result can be found in \cite{bb}.
\begin{theorem}
Given a function $S\in{\mathcal {S}}_d(\cU, \cY)$, let
$V$ be the isometric operator defined in \eqref{2.5}. Then
\begin{enumerate}
\item  A block-operator matrix $\bU$ of the form \eqref{2.8}  
is a $\cfm$ colligation for $S$ if   
and only if $\bU^*$  is a contractive extension of $V$ to all of
$\cH(K_S)^d\oplus\cY$, i.e.,
\begin{equation}
\bU^*\vert_{{\mathcal D}_V}=V\quad\mbox{and}\quad \|\bU^*\|\le 1.
\label{2.9}
\end{equation}
In particular, a $\cfm$ colligation for $S$ exists.
\item Every $\cfm$ colligation $\bU$ for $S$
is weakly coisometric and observable and
furthermore, $S(z) = D+C(I-Z_{\cH(K_S)}(z)A)^{-1}Z_{\cH(K_S)}(z) B$.
\item Any observable weakly coisometric colligation $\widetilde{\bU}$
of the form \eqref{1.16} with the  characteristic  function equal
$S$ is unitarily equivalent to some $\cfm$ colligation $\bU$ for $S$.
\end{enumerate}
\label{T:2.4}
\end{theorem}

\begin{remark}
{\rm Since $Z_{\cH(K_{S})}(0)=0$, the space ${\mathcal D}_V$ contains
all 
vectors of the form $\sbm{0 \\ y}$ and therefore it splits into the
direct 
sum ${\mathcal D}_V={\mathcal D}\oplus\cY$ where 
$$
{\mathcal D}=\bigvee_{\zeta\in\B^d, \, y\in
\cY}Z_\cY(\zeta)^*K_S(\cdot,\zeta)y\subset \cH(K_S)^d.
$$
It is readily checked that the orthogonal complement ${\mathcal
D}^\perp
=\cH(K_S)^d\ominus {\mathcal D}$ of ${\mathcal D}$ is given by
\begin{equation}
{\mathcal D}^\perp=\left\{h\in\cH(K_S)^d: \; Z_\cY(z)h(z)\equiv
0\right\}.
\label{2.9a}
\end{equation}
We now see from \eqref{2.9} that nonuniqueness of $\cfm$ colligations
for $S$ is achieved by different choices of $A^*\vert_{{\mathcal 
D}^\perp}$ and of $B^*\vert_{{\mathcal D}^\perp}$.}
\label{R:2.3}
\end{remark}

\begin{remark}
{\rm Observe also that in the univariate case $d=1$, the operators
$A$ and 
$B$ are uniquely recovered from \eqref{2.6}, \eqref{2.7} and one then 
arrives at the operators $A$ and $B$ exactly as in 
\eqref{1.8}. Also, the space ${\mathcal D}$ equals $\cH(K_S)$
so that $\bU^*=V$ and now it is seen that for $d=1$ Theorem 
\ref{T:2.4} collapses to part (1) in Theorem \ref{T:sum}.}
\label{R:2.4}
\end{remark} 

Definition \ref{D:2.3} does not require $\bU$ to be a realization for
$S$: representation \eqref{1.12} is automatic once the operators
$A$, $B$, $C$ and $D$ are of the required form. Theorem \ref{T:2.5a} 
below(see Theorem 2.10 in \cite{bb} for the proof) characterizes 
which operators $A$ and which operators $B$ can arise in a 
\textbf{c.f.m.}~colligation for $S$. Let us say that $A \colon
\cH(K_S)\to\cH(K_S)^d$ is a {\em contractive solution} of the Gleason
problem for $\cH(K_S)$ if in addition to \eqref{2.6}, the inequality
$$
\sum_{k=1}^{d} \|A_kf\|^{2}_{\cH(K_{S})}
\le\|f\|^{2}_{\cH(K_{S})} - \| f(0)\|^{2}_{\cY}
$$
holds for every $f\in\cH(K_S)$. An equivalent operator form of this 
inequality is $A^*A+C^*C\le I$
where the operator $C \colon \cH(K_S)\to \cY$ is given in
\eqref{2.7a}. It therefore follows from Definition \ref{D:2.3} that
for every {\bf c.f.m.}~colligation $\bU=\left[
\begin{smallmatrix} A & B \\ C & D \end{smallmatrix} \right]$ for
$S$, the
operator $A$ is a contractive solution of the Gleason problem
\eqref{2.6}.
\begin{theorem}
Let $S\in\cS_d(\cU,\cY)$ be given and let us assume that $C$, $D$ are
given by formulas \eqref{2.7a}. Then
\begin{enumerate}
\item For every contractive solution $A$ of the Gleason problem
\eqref{2.6}, there exists an operator $B: \; \cU\to\cH(K_S)$ such
that 
$\bU =\left[ \begin{smallmatrix} A & B \\ C & D \end{smallmatrix}
\right]$
is contractive and $S$ is realized as in \eqref{1.12}.
\item Every such $B$ solves the $\cH(K_S)$-Gleason problem \eqref{2.7}
so that $\bU$ is a {\bf c.f.m.} colligation.
\end{enumerate}
\label{T:2.5a}   
\end{theorem}   

An object of an independent interest is the class of Schur-class functions
admitting contractive {\em commutative} realizations of the form \eqref{1.12} where
the state space operators $A_1,\ldots,A_d$ commute with each
other. A key role here is played by the backward shift on the Drury-Arveson space
$\cH(k_d)$, the commuting $d$-tuple ${\mathbf  M}^*_{\bz}:=(M^*_{z_1},\ldots,M^*_{z_d})$
consisting of the adjoints (in metric of $\cH(k_d)$) of operators $M_{z_j}$'s of
multiplication by the coordinate functions of $\C^d$. It was shown in \cite{BBF2b} that
any Schur-class function $S$ with associated de Branges-Rovnyak space $\cH(K_S)$
finite-dimensional and not ${\mathbf  M}^*$-invariant does not admit a contractive
commutative realization. The following theorem also can be found in \cite{BBF2b}.

\begin{theorem}  \label{T:comreal}
A Schur-class function $S\in\cS_d(\cU,\cY)$ admits a commutative weakly
coisometric realization if and only if the following conditions hold:
\begin{enumerate}
\item The
associated de Branges-Rovnyak space $\cH(K_S)$ is ${\bf M}_\bz^*$-invariant,
and
\item the inequality
\begin{equation}
\sum_{j=1}^d\|M_{z_j}^*f\|^2_{\cH(K_S)}\le
\|f\|^2_{\cH(K_S)}-\|f(0)\|^2_{\cY}\quad\mbox{holds for all}\quad f \in\cH(K_{S}).
\label{3.1}
\end{equation}
\end{enumerate}
Furthermore, if conditions (1) and (2) are satisfied, then there exists a commutative {\bf
c.f.m.} colligation for $S$. Moreover, the state-space operators tuple is equal to
the Drury-Arveson backward shift restricted to $\cH(K_S)$: $A_j=M^*_{z_j}\vert_{\cH(K_S)}$
for
$j=1,\ldots,d$.
\end{theorem}

\section{Weakly isometric realizations} 
\label{T}

In the univariate case, the state space of the functional-model
isometric 
realization for a Schur-class 
function $S$ can be taken to be equal to the reproducing kernel
Hilbert 
space
$\cH(\widetilde{K}_S)$ with reproducing kernel 
$\widetilde{K}_S(z,\zeta)$ as in \eqref{1.1}.
A natural multivariable counterpart of this kernel would be the kernel
$$
\widetilde{K}_S(z,\zeta)=\frac{I_{\cU}-S(z)^*S(\zeta)}{1-\langle
\zeta, 
z\rangle}.
$$
However, if $S\in\cS_d(\cU,\cY)$ for $d>1$, this kernel is not 
positive in general.
Instead, we have the following Agler-type decomposition result (see 
\cite[Theorem 2.4]{BTV} for the proof).
\begin{theorem}
A function $S \colon {\mathbb B}^{d} \to \cL(\cU, \cY)$ belongs to 
$\cS_{d}(\cU,\cY)$ if and only if there exists a positive kernel
\begin{equation}
\Phi=\begin{bmatrix}\Phi_{11}&\ldots&
\Phi_{1d} \\ \vdots && \vdots \\ \Phi_{d1}&\ldots&
\Phi_{dd}\end{bmatrix}: \; \B^d\times\B^d\to \cL(\cU^d)
\label{2.10}
\end{equation}
so that for every $z,\zeta\in\B^d$,
\begin{equation}
I_{\cU}-S(z)^{*}S(\zeta) =
\sum_{j=1}^{d}\Phi_{jj}(z,\zeta) -
\sum_{i,\ell=1}^{d} \overline{z}_i\zeta_\ell \Phi_{i\ell}(z,\zeta).
\label{2.11}
\end{equation}  
\label{T:2.5}
\end{theorem}
The kernel $\Phi$ in Theorem \ref{T:2.5} is not determined from $S$ 
uniquely. With each such kernel, one can associate weakly-isometric 
functional-model colligations as follows. Given a decomposition 
\eqref{2.11} with a positive kernel $\Phi$, let $\cH(\Phi)$ be the 
reproducing kernel Hilbert space with reproducing kernel $\Phi$.
Clearly, the elements of $\cH(\Phi)$ are $\cU^d$-valued functions
defined 
and conjugate-analytic on $\B^d$. We next rearrange the block 
columns $\Phi_{\bullet k}$ of $\Phi$ to produce the kernel
\begin{equation}
{\mathbb T}(z,\zeta):= \left[\begin{array}{c}
\Phi_{\bullet 1}(z,\zeta) \\ \vdots \\
\Phi_{\bullet d}(z,\zeta)\end{array}\right], \quad\mbox{where}\quad
\Phi_{\bullet k}(z,\zeta)=\left[\begin{array}{c}\Phi_{1k}(z,\zeta) \\ 
\vdots \\ \Phi_{dk}(z,\zeta)\end{array}\right],
\label{2.12}
\end{equation}
and we then introduce the subspace
\begin{equation}
\widetilde{\mathcal D}=\bigvee\left\{\sum_{j=1}^d 
\zeta_j\Phi_{\bullet j}(\cdot,\zeta)u: \; 
\zeta=(\zeta_1,\ldots,\zeta_d)\in\B^d, \,
u\in\cU\right\}\subset\cH(\Phi).
\label{2.13}  
\end{equation}
Decomposition \eqref{2.11} then can be written in the inner product
form 
as
\begin{align*}
&\left\langle \left[\begin{array}{c}
\sum_{j=1}^d \zeta_j\Phi_{\bullet j}(\cdot,\zeta)u \\ 
u\end{array}\right], \;
\left[\begin{array}{c} \sum_{j=1}^d
z_j\Phi_{\bullet j}(\cdot,z)u'\\
u'\end{array}\right]\right\rangle_{\cH(\Phi)\oplus\cU}\notag\\
&\qquad=\left\langle \left[\begin{array}{c}
{\mathbb T}(\cdot,\zeta)u \\ S(\zeta)u\end{array}\right], \;
\left[\begin{array}{c}{\mathbb T}(\cdot,z)u' \\
S(z)u'\end{array}\right]\right\rangle_{\cH(\Phi)^d\oplus\cY}
\end{align*}
so that the linear map $\widetilde{V}$ given by formula 
\begin{equation}
\widetilde{V}
\colon  \;\left[\begin{array}{c}\sum_{j=1}^d \zeta_j\Phi_{\bullet 
j}(\cdot,\zeta)u\\ u\end{array}\right] \to
\left[\begin{array}{c}{\mathbb T}(\cdot,\zeta)u
 \\ S(\zeta)u\end{array}\right]
\label{2.14}
\end{equation}
extends by continuity to define the isometry $\widetilde{V}: \,
{\mathcal D}_{\widetilde{V}}\to {\mathcal R}_{\widetilde{V}}$ where
$$
{\mathcal
D}_{\widetilde{V}}=\widetilde{\mathcal{D}}\oplus\cU\quad\mbox{and}\quad{\mathcal
R}_{\widetilde{V}}=\bigvee_{\zeta\in\B^d, \, u\in \cU}
\left[\begin{array}{c}
{\mathbb T}(\cdot,\zeta)y\\ S(\zeta)u\end{array}\right]
\subset \begin{bmatrix}\cH(\Phi)^d \\ \cY\end{bmatrix}.
$$
\begin{definition}
{\rm Given a function $S \in\cS_d(\cU, \cY)$, we shall say
that the contractive block-operator matrix
\begin{equation}
\widetilde{\bU} = \begin{bmatrix} \tA & \tB \\ \tC &
\tD\end{bmatrix}\colon \;
\begin{bmatrix}\cH(\Phi) \\ \cU\end{bmatrix}\to 
\begin{bmatrix}\cH(\Phi)^d \\ \cY\end{bmatrix}
\label{2.15}
\end{equation}
is a {\em dual canonical functional-model ({\rm abbreviated to $\dcfm$
 in what follows}) colligation} associated with 
decomposition \eqref{2.11} for $S$ if:
\begin{enumerate}
\item The restrictions of operators $A$ and $C$ to the subspace
$\widetilde{\mathcal D}\subset \cH(\Phi)^d$
defined in \eqref{2.13} have the following action on special kernel 
functions:
\begin{align*}
\tA\vert_{\widetilde{\mathcal D}}: & \; 
\sum_{j=1}^d \zeta_j\Phi_{\bullet j}(\cdot,\zeta)u
\to {\mathbb T}(\cdot,\zeta)u-{\mathbb T}(\cdot,0)u,\\
\tC\vert_{\widetilde{\mathcal D}}: \; & 
\sum_{j=1}^d \zeta_j\Phi_{\bullet j}(\cdot,\zeta)u
\to S(\zeta)u-S(0)u.
\end{align*}
\item The operators $\tB\colon \, \cU\to\cH({\mathbb K}_R)^q$
and $\tD\colon\cU\to \cY$ are given by
$$
\tB \colon u\mapsto {\mathbb T}(\cdot,0)u, \qquad \tD \colon u \mapsto
S(0) u.
$$
\end{enumerate}}
\label{D:3.1}
\end{definition}

The following theorem is parallel to Theorem \ref{T:2.4}; note that 
when $d=1$ this theorem amounts to part (2) of Theorem \ref{T:sum}.

\begin{theorem}
Given a function $S\in{\mathcal {S}}_d(\cU, \cY)$ with a 
a fixed decomposition \eqref{2.11}, let
$\widetilde{V}$ be the isometric operator defined in \eqref{2.14}.
Then
\begin{enumerate}
\item  A block-operator matrix $\bU$ of the form \eqref{2.15}
is a $\dcfm$ colligation for $S$ if
and only if $\bU$  is a contractive extension of $\widetilde{V}$ to
all of
$\cH(\Phi)^d\oplus\cY$. In particular, a $\dcfm$ colligation for $S$ 
exists.
\item Every $\dcfm$ colligation $\bU$ for $S$
is weakly isometric and controllable and
furthermore, $S(z) = D+C(I-Z_{\cH(\Phi)}(z)A)^{-1}Z_{\cH(\Phi)}(z) B$.
\item Any controllable weakly isometric colligation $\widetilde{\bU}$ 
\eqref{1.16} with its  characteristic  function equal
$S$ is unitarily equivalent to some $\dcfm$ colligation for $S$
based on the Agler decomposition \eqref{2.11} with 
$\Phi_{i \ell}(z, \zeta)$ given by
$$
\Phi_{i\ell}(z,\zeta) 
=\widetilde{B}^*(I-Z_{\cX}(z)^*\widetilde{A}^*)^{-1}
{\mathcal I}_{i}^*{\mathcal I}_\ell
(I-\widetilde{A}Z_\cX(\zeta))^{-1}\tB.
$$
\end{enumerate}
\label{T:2.7}
\end{theorem}

The proof can be extracted from the proof of Theorem \ref{T:4.7} in
the 
next section and will be omitted.

\section{Weakly unitary realizations}
\label{S:unit}

In this section we will construct functional model weakly-unitary 
realizations for functions $S\in {\mathcal {S}}_{d}(\cU, \cY)$.
The state space for these realizations will be the reproducing kernel 
Hilbert space with reproducing kernel ${\mathbb K}$ which
incorporates 
both 
kernels $K_S$ and $\Phi$ from the previous section. In what follows, 
$\{{\bf e}_1,\ldots,{\bf e}_d\}$ stands for the standard bases for
$\C^d$
and we use notation
\begin{equation}
M(z) =\begin{bmatrix} I_\cY & 0 \\ 0 & Z_{\cU}(\bar{z})^*\end{bmatrix}
\quad\mbox{and}\quad
N_j(z) =\begin{bmatrix} \bar{z}_j I_\cY  & 0 \\
0 & {\bf e}_j \otimes I_{\cU}\end{bmatrix}\quad (j=1,\ldots,d).
\label{4.1}
\end{equation}

\begin{theorem}   
\label{T:3.1}
Let $S\in\cS_{d}(\cU,\cY)$ and let $M(z)$ and $N_k(z)$ be defined as
in 
\eqref{4.1}.  The kernel $K_S$ defined in \eqref{1.10}
can be extended to the positive kernel 
\begin{equation}
{\mathbb K}(z,\zeta)=\begin{bmatrix} K_{S}(z,\zeta) & 
\Psi_1(z,\zeta) & \dots & \Psi_d(z,\zeta) \\ 
\Psi_1(\zeta,z)^* & \Phi_{11}(z,\zeta)& \dots & \Phi_{1d}(z,\zeta)\\
\vdots & \vdots &&\vdots \\
\Psi_d(\zeta,z)^* & \Phi_{d1}(z,\zeta) & \dots & \Phi_{dd}(z,\zeta)
\end{bmatrix}: \; \B^d\times\B^d\to \cL(\cY\oplus\cU^d)
\label{4.2}   
\end{equation}
subject to identity 
\begin{align}
&\begin{bmatrix} I_{\cY} \\ S(z)^{*} \end{bmatrix} \begin{bmatrix}
I_{\cY} & S(\zeta) \end{bmatrix} -
\begin{bmatrix} S(z) \\ I_{\cU} \end{bmatrix} \begin{bmatrix}
S(\zeta)^{*} & I_{\cU}
\end{bmatrix} \notag \\
& =M(z)^{*} {\mathbb K}(z,\zeta)M(\zeta)-\sum_{j=1}^{d} N_{j}(z)^{*}
{\mathbb K}(z,\zeta)N_{j}(\zeta).\label{4.3}
\end{align}
\end{theorem}

\begin{proof}
    By Theorem \ref{T:1.1} we know that any $S \in \cS_{d}(\cU, \cY)$ 
    can be realized as in \eqref{1.12} with $\bU$ as in \eqref{1.11} 
    unitary.  It is then a straightforward calculation to show that 
    the kernel
\begin{equation}  \label{4.3'}
{\mathbb K}(z,\zeta) =  {\mathbb G}(z) {\mathbb G}(\zeta)^{*}
\end{equation}
with
$$
{\mathbb G}(z):=\left[\begin{array}{c}
C (I-Z_{\cX}(z) A)^{-1}\\
B^*(I-Z_{\cX}(z)^* A^*)^{-1}{\mathcal 
I}_{1}\\ \vdots \\ 
B^*(I-Z_{\cX}(z)^* A^*)^{-1}{\mathcal
I}_{d} \end{array}\right]: \; \B^d\to \cL(\cX, \cY\oplus\cU^d)
$$
provides a positive-kernel solution of the identity \eqref{4.3}.
Note that ${\mathbb K}(z, \zeta)$ in \eqref{4.3'} has the form 
\eqref{4.2} with 
\begin{align}
&  \Psi_{k}(z, \zeta) = C (I - Z_{\cX}(z) A)^{-1} {\mathcal I}_{k}^{*}
(I - A Z_{\cX}(\zeta))^{-1} B,  \label{4.3''}\\
& \Phi_{ij}(z, \zeta) = B^{*}(I - Z_{\cX}(z)^{*}A^{*})^{-1} {\mathcal 
I}_{i} {\mathcal I}_{j}^{*} (I - A Z_{\cX}(z))^{-1} B.  \label{4.3'''}
\end{align}
\end{proof}

Identity \eqref{4.3} (as well as the kernel ${\mathbb K}$ itself) 
will be called an {\em Agler decomposition} for $S$. 
Equating the diagonal block entries in 
\eqref{4.3} one gets 
\eqref{2.4} and \eqref{2.11}; equality of nondiagonal blocks gives 
\begin{equation}
S(z)-S(\zeta)=\sum_{j=1}^{d}(z_j-\zeta_j)\Psi_j(z,\zeta).
\label{4.4}
\end{equation}
We let $\cH({\mathbb K})$ be the reproducing kernel Hilbert space 
associated with the kernel ${\mathbb K}$ and remark that 
the elements of  $\cH({\mathbb K})$ are the $\cY\oplus\cU^d$-valued 
functions of the form
\begin{equation}
f=\begin{bmatrix}f_{+} \\ f_{-}\end{bmatrix}: \, \B^d\to
\begin{bmatrix}\cY \\ \cU^d\end{bmatrix},\quad
f_-=\bigoplus_{i=1}^d f_{-,i},
\label{4.5}
\end{equation}
where $f_+$ is analytic and $f_-$ is conjugate-analytic on ${\mathbb 
B}^{d}$. For
functions $g\in\cH({\mathbb K})^d$, we will use the following
representation and notation
\begin{equation}
g=\bigoplus_{i=1}^d g_i:=\left[\begin{array}{c}g_1 \\
\vdots \\ g_d\end{array}\right]:\quad
g_i=\begin{bmatrix}g_{i,+} \\ g_{i,-}\end{bmatrix}\in\cH({\mathbb K}),
\; \; g_{i,-}=\bigoplus_{j=1}^d g_{i,-,j}.
\label{4.6}
\end{equation}
We next observe  that Agler decomposition \eqref{4.3}
can be written in the inner product form as 
the identity
\begin{align}
&\langle y+S(\zeta)u, \; y'+S(z)u'\rangle_{\cY}-\langle
S(\zeta)^*y+u, \;
S(z)^*y'+u'\rangle_{\cU}\notag\\
&=\left\langle {\mathbb K}(\cdot,\zeta)M(\zeta)
\begin{bmatrix}y \\ u\end{bmatrix}, \, {\mathbb
K}(\cdot,z)M(z)\begin{bmatrix}y' \\ u'\end{bmatrix}\right
\rangle_{\cH({\mathbb K})}\notag\\
&\quad -\left\langle\bigoplus_{j=1}^d {\mathbb
K}(\cdot,\zeta)N_{j}(\zeta)\begin{bmatrix}y \\ u\end{bmatrix}, \,
\bigoplus_{j=1}^d{\mathbb K}(\cdot,z)N_{j}(z)\begin{bmatrix}y' \\
u'\end{bmatrix}\right\rangle_{\cH({\mathbb K})^d}
\label{4.7}
\end{align}
holding for all $z,\zeta\in\B^d$, $y,y'\in\cY$ and $u,u'\in\cU$. We
next denote by ${\mathbb K}_0, {\mathbb K}_1,\ldots,{\mathbb
K}_d$ the block  columns of the kernel \eqref{4.2}:
\begin{equation}
{\mathbb K}_0(z,\zeta)=\begin{bmatrix}K_S(z,\zeta) \\
\Psi_1(\zeta,z)^*\\ \vdots \\ \Psi_d(\zeta,z)^*\end{bmatrix},\quad
{\mathbb K}_j(z,\zeta)=\begin{bmatrix}\Psi_j(z,\zeta) \\
\Phi_{1j}(z,\zeta)\\ \vdots \\ \Phi_{dj}(z,\zeta)\end{bmatrix}\quad
(j=1,\ldots,d)
\label{4.8}   
\end{equation}
and use them to define a new  kernel
\begin{equation}
{\mathbb T}(z,\zeta):=\begin{bmatrix}{\mathbb K}_1(z,\zeta)\\
\vdots \\ {\mathbb K}_d(z,\zeta)\end{bmatrix}\colon\B^d\times\B^d\to
\cL(\cU,(\cY\oplus\cU^d)^d).
\label{4.9}   
\end{equation}
The relations
\begin{align*}
{\mathbb K}(\cdot,\zeta)M(\zeta)\begin{bmatrix}y \\
u\end{bmatrix}&={\mathbb 
K}_0(\cdot,\zeta)y+\sum_{j=1}^d \zeta_j{\mathbb K}_j(\cdot,\zeta)u,\\
{\mathbb K}(\cdot,\zeta)N_{j}(\zeta)\begin{bmatrix}y \\
u\end{bmatrix}&=\overline{\zeta}_j{\mathbb
K}_0(\cdot,\zeta)y+{\mathbb K}_j(\cdot,\zeta)u,\end{align*} 
follow immediately from \eqref{4.1} and \eqref{4.2} and allow us to 
rewrite \eqref{4.7} as
\begin{align}
&\left\langle \left[\begin{array}{c}
{\bQ}(\zeta)^*\otimes{\mathbb K}_0(\cdot,\zeta)y
+{\mathbb T}(\cdot,\zeta)u \\
y+S(\zeta)u\end{array}\right], \;
 \left[\begin{array}{c}{\bQ}(z)^*\otimes{\mathbb K}_0(\cdot,z)y'
+{\mathbb T}(\cdot,z)u' \\ y'+S(z)u'\end{array}\right]
\right\rangle\notag\\
&=\left\langle \left[\begin{array}{c}
{\mathbb K}_0(\cdot,\zeta)y+\sum_{j=1}^d\zeta_j{\mathbb
K}_j(\cdot,\zeta)u
\\ S(\zeta)^*y+u\end{array}\right], \;
\left[\begin{array}{c}
{\mathbb K}_0(\cdot,z)y'+\sum_{j=1}^d z_j{\mathbb K}_j(\cdot,z)u'\\
S(z)^*y'+u'\end{array}\right]\right\rangle.
\label{4.10}
\end{align}

\begin{lemma}
Let ${\mathbb K}$ be a fixed Agler decomposition for a function 
$S\in\cS_d(\cU,\cY)$ and let ${\mathbb K}_j$ and ${\mathbb T}$ be
given by 
\eqref{4.8}, \eqref{4.9}. Then the map 
\begin{align} 
V: \; \left[\begin{array}{c}
{\bQ}(\zeta)^*\otimes{\mathbb K}_0(\cdot,\zeta)y
+{\mathbb T}(\cdot,\zeta)u \\ y+S(\zeta)u\end{array}\right]
\to\left[\begin{array}{c}
{\mathbb K}_0(\cdot,\zeta)y+\sum_{j=1}^d 
\zeta_j{\mathbb K}_j(\cdot,\zeta)u
\\ S(\zeta)^*y+u\end{array}\right].
\label{4.11} 
\end{align}
extends by linearity and continuity to an isometry from 
\begin{equation}
{\mathcal D}_V={\mathcal D}\oplus\cY\quad\mbox{onto}\quad
{\mathcal R}_V={\mathcal R}\oplus\cU
\label{4.12}
\end{equation}
where the subspaces ${\mathcal D}\subset \cH({\mathbb K})^d$ and 
${\mathcal R}\subset \cH({\mathbb K})$ are given by
\begin{align}
{\mathcal D}&=\bigvee\left\{
{\bQ}(\zeta)^*\otimes{\mathbb K}_0(\cdot,\zeta)y, \; \;
{\mathbb T}(\cdot,\zeta)u: \, \zeta\in\B^d, \, y\in\cY, \,
u\in\cU\right\},\label{4.13}\\
{\mathcal R}&=\bigvee\left\{{\mathbb K}_0(\cdot,\zeta)y, \; \;
\sum_{j=1}^d \zeta_j{\mathbb K}_j(\cdot,\zeta)u:
\, \zeta\in\B^d, \, y\in\cY, \, u\in\cU\right\}. 
\label{4.14}
\end{align}
\label{L:3.2}
\end{lemma}

\begin{proof} It follows from \eqref{4.10} that $V$ defined as in 
\eqref{4.11} extends by linearity and continuity to an isometry from 
$$
{\mathcal D}_V=\bigvee\left\{
\left[\begin{array}{c}{\bQ}(\zeta)^*\otimes{\mathbb
K}_0(\cdot,\zeta)y\\
y\end{array}\right], \; \left[\begin{array}{c}
{\mathbb T}(\cdot,\zeta)u \\ S(\zeta)u\end{array}\right]
: \, \zeta\in\B^d, \, y\in\cY, \, u\in\cU\right\}
$$
onto
$$
{\mathcal R}_V=\bigvee\left\{
\left[\begin{array}{c}
{\mathbb T}(\cdot,\zeta)y\\ S(\zeta)^*y\end{array}\right], \;
\left[\begin{array}{c}\sum_{j=1}^d \zeta_j{\mathbb K}_j(\cdot,\zeta)u
\\ u\end{array}\right]:
\, \zeta\in\B^d, \, y\in\cY, \, u\in\cU\right\}.
$$
It is readily seen that ${\mathcal D}_V$ and ${\mathcal R}_V$ contain
respectively all vectors of the form $\sbm{0 \\ y}$ and $\sbm{0 \\ u}$
and therefore they split into the direct sums \eqref{4.12}.
\end{proof}
A straightforward verification shows that the
orthogonal complements ${\mathcal D}^\perp=\cH({\mathbb
K})^d\ominus{\mathcal D}$ and ${\mathcal R}^\perp=\cH({\mathbb
K})\ominus{\mathcal R}$ (the defect spaces of the isometry $V$)
can be described as
\begin{align}
{\mathcal D}^\perp&=\left\{g\in
\cH({\mathbb K})^d: \; \sum_{i=1}^dz_ig_{i,+}(z)\equiv 0 \; \;
\text{and} 
\; \;  \sum_{i=1}^d g_{i,-,i}(z)\equiv 0\right\},\label{4.15}\\
{\mathcal R}^\perp&=\left\{
f\in\cH({\mathbb K}): \; f_+(z)\equiv 0 \; \; \text{and} \; \;
\sum_{i=1}^d \overline{z}_if_{-,i}(z)\equiv 0\right\}
\label{4.16}
\end{align}  
where we have used notation \eqref{4.5} and \eqref{4.6}.
We next use the same notation to define two linear maps ${\bf s}: \, 
\cH({\mathbb K})\to
\cH(K_S)$ and $\widetilde{\bf s}: \, \cH({\mathbb K})^d\to
\cH(\Phi)$ by
\begin{equation}
{\bf s}: \; f\mapsto f_{+},\qquad
\widetilde{\bf s}: \;  g=\bigoplus_{i=1}^dg_i \mapsto \sum_{i=1}^d
g_{i,-,i},
\label{4.17}
\end{equation}  
and observe the equalities
\begin{equation}
\langle f, \, {\mathbb K}_0(\cdot,\zeta)y\rangle_{\cH({\mathbb K})}
=\langle ({\bf s}f)(\zeta), \, y\rangle_{\cY},\qquad
\langle g, \, {\mathbb T}(\cdot,\zeta)u\rangle_{\cH({\mathbb
K})^d}=\langle (\widetilde{\bf s}g)(\zeta), \, u\rangle_{\cU}
\label{4.18}
\end{equation}
holding for all $f\in\cH({\mathbb K})$, $g\in\cH({\mathbb K})^d$,
$\zeta\in\B^d$, $y\in\cY$ and $u\in\cU$. Indeed, for a
function $f$ in $\cH({\mathbb K})$, we have from \eqref{4.8} by the
reproducing kernel property
$$
\langle f, \, {\mathbb K}_0(\cdot,\zeta)y\rangle_{\cH({\mathbb K})}=
\left\langle f, \; {\mathbb K}(\cdot,\zeta)\begin{bmatrix}
y \\ 0 \end{bmatrix}\right\rangle_{\cH({\mathbb K})}=
\left\langle f_{+}(\zeta), \, y
\right\rangle_{\cY}=\langle ({\bf s}f)(\zeta), \, y\rangle_{\cY}
$$
which proves the first equality in \eqref{4.18}. The proof of the
second
is much the same. 
\begin{definition}
A contractive colligation 
\begin{equation}
\bU = \begin{bmatrix} A & B \\ C & D\end{bmatrix}\colon \;
\begin{bmatrix}\cH({\mathbb K}) \\ \cU\end{bmatrix}\to
\begin{bmatrix}\cH({\mathbb K})^d \\ \cY\end{bmatrix}
\label{4.19}
\end{equation}
will be called {\em a two-component canonical functional-model}
({\rm abbreviated to $\tcfm$ in what follows})
colligation associated with a fixed Agler decomposition \eqref{4.3}
of a 
given $S\in{\mathcal {S}}_{d}(\cU, \, \cY)$ if
\begin{enumerate}
\item The state space operator $A=\operatorname{Col}_{1\le k\le d}A_k$
solves the structured Gleason problem
\begin{equation}
({\bf s}f)(z)-({\bf s}f)(0)=\sum_{k=1}^d z_k
(A_kf)_+(z)\quad\mbox{for 
all}\quad f\in  \cH({\mathbb K}),
\label{4.20}
\end{equation}
whereas the adjoint operator $A^*$
solves the dual structured Gleason problem
\begin{equation}  
(\widetilde{\bf s}g)(z)-(\widetilde{\bf s}g)(0)=
\sum_{k=1}^q\overline{z}_k(A^*g)_{-,k}(z)\quad\mbox{for    
all}\quad g\in  \cH({\mathbb K})^{d}.
\label{4.21}  
\end{equation}
\item The operators $C: \, \cH({\mathbb K})\to \cY$, $B^*:
 \cH({\mathbb K})^d\to \cU$ and $D: \, \cU\to\cY$ are of the form 
\begin{equation}
C\colon \, f\to ({\bf s}f)(0),\quad B^*\colon \, g\to(\widetilde{\bf 
s}g)(0)\quad\mbox{and}\quad D\colon u\to S(0)u.
\label{4.22}
\end{equation}
\end{enumerate}
\label{D:4.1}
\end{definition}

\begin{proposition}
Relations \eqref{4.19}, \eqref{4.21} and \eqref{4.22} are equivalent 
respectively to equalities
\begin{align}
& A^*\left({\bQ}(\zeta)^*\otimes{\mathbb K}_0(\cdot,\zeta)y\right)=
{\mathbb K}_0(\cdot,\zeta)y-{\mathbb K}_0(\cdot,0)y,\label{4.23}\\
&A\left(\sum_{j=1}^d \zeta_j {\mathbb K}_j(\cdot,\zeta)u\right)
={\mathbb T}(\cdot,\zeta)u-{\mathbb T}(\cdot,0)u,\label{4.24}\\
& C^*y={\mathbb K}_0(\cdot,0)y,\quad
Bu={\mathbb T}(\cdot,0)u,\quad\mbox{and}\quad D^*y=S(0)^*y
\label{4.25}
\end{align}
holding for every $\zeta\in\B^d$, $y\in\cY$ and $u\in\cU$.
\label{P:4.2}
\end{proposition}

\begin{proof} It follows from the first equality in \eqref{4.18} that 
$$
\left\langle ({\bf s}f)(z)-({\bf s}f)(0), \, y\right\rangle_\cY=
\left\langle f, \, {\mathbb K}_0(\cdot,z)y -{\mathbb 
K}_0(\cdot,0)y\right\rangle_{\cH({\mathbb K})}
$$
and on the other hand, 
\begin{align*}
\left\langle \sum_{k=1}^d (A_kf)_+(z),  \, y\right\rangle_\cY&=
\sum_{k=1}^d \left\langle A_kf, \, \overline{z}_k{\mathbb
K}_0(\cdot,z)y\right\rangle_{\cH({\mathbb K})}\notag \\
&=\left\langle Af, \, {\bQ}(z)^*\otimes{\mathbb 
K}_0(\cdot,z)y\right\rangle_{\cH({\mathbb K})^d}\notag\\
&=\left\langle f, \, A^*\left({\bQ}(z)^*\otimes{\mathbb 
K}_0(\cdot,z)y\right)\right\rangle_{\cH({\mathbb K})}.
\end{align*}
Since the two latter equalities hold for every 
$f\in \cH({\mathbb K})$ and $y\in\cY$, the equivalence 
\eqref{4.20} $\Leftrightarrow$ \eqref{4.23} follows. 
The equivalence \eqref{4.21}$\Leftrightarrow$ \eqref{4.24} follows 
from \eqref{4.18} in much the same way;
the formula for $C^*$ in \eqref{4.24} follows from 
$$
\langle f, \, C^*y\rangle=\langle Cf, \, y\rangle=\langle ({\bf
s}f)(0), 
\, y\rangle=\left\langle f, \, {\mathbb K}_0(\cdot,0)y\right\rangle
$$
and the formula for $B$ is a consequence of a similar computation.
The formula for $D^*$ is self-evident.\end{proof}
\begin{proposition}
Let $B$, $C$ and $D$ be the operators defined in \eqref{4.22}. Then
\begin{equation}
CC^*+DD^*=I_{\cY} \quad\mbox{and}\quad B^*B+D^*D=I_{\cY}.
\label{4.26}
\end{equation}
Furthermore,
\begin{align}
B^*\colon & {\bQ}(\zeta)^*\otimes{\mathbb K}_0(\cdot,\zeta)y
\to S(\zeta)^*y-S(0)^*y,\label{4.27}\\    
B^*\colon & {\mathbb T}(\cdot,\zeta)u
\to u-S(0)^*S(\zeta)u\label{4.28}
\end{align}
for all $\zeta\in\B^d$, $y\in\cY$ and $u\in\cU$, where ${\mathbb K}_0$
and ${\mathbb T}$ are defined in \eqref{4.8}, \eqref{4.9}.
\end{proposition}
\begin{proof} Upon letting $f={\mathbb K}_0(\cdot,\zeta)y$ and 
$g={\mathbb T}(\cdot,\zeta)u$ in formulas \eqref{4.18} and making use
of \eqref{4.8} we get
\begin{align}
\left\langle {\mathbb K}_0(\cdot,\zeta)y, \,
{\mathbb K}_0(\cdot,z)y\right\rangle_{\cH({\mathbb K})^p}&=
\left\langle K_{S}(z,\zeta)y, \, y\right\rangle_{\cY},
\label{4.29}\\
\left\langle {\mathbb T}(\cdot,\zeta)u, \, {\mathbb T}(\cdot,z)u
\right\rangle_{\cH({\mathbb K})^d}&=\sum_{j=1}^d
\left\langle\Phi_{jj}(z,\zeta)u, \,
u\right\rangle_{\cU}.\label{4.30}
\end{align}
We then have
\begin{align*}
\|C^*y\|^2&=\left\|{\mathbb K}_0(\cdot,0)y \right\|^2=\left\langle 
K_S(0,0)y, \, y\right\rangle=\left\langle (I-S(0)S(0)^*)y, \, 
y\right\rangle,\\
\|Bu\|^2&=\left\|{\mathbb T}(\cdot,0)u\right\|^2=
\left\langle \sum_{k=1}^d  \Phi_{kk}(0,0) u, \, 
u\right\rangle=\left\langle (I-S(0)^*S(0))u, \, u\right\rangle,
\end{align*}
where the first equalities follow from formulas \eqref{4.25} for $B$
and 
$C^*$, the second equalities follow upon letting $z=\zeta=0$ in 
\eqref{4.29}, \eqref{4.30}, and finally, the third equalities follow
from 
the representation  formulas \eqref{2.11} and \eqref{2.12} evaluated
at 
$z=\zeta=0$. Taking into account formulas \eqref{4.22} 
and \eqref{4.25} for $D$ and $D^*$, we then have equalities
\begin{align}
\|C^*y\|^2&=\|y\|^2-\|S(0)^*y\|^2=\|y\|^2-\|D^*y\|^2,\label{4.31}\\
\|Bu\|^2&=\|u\|^2-\|S(0)u\|^2=\|u\|^2-\|Du\|^2\notag
\end{align}
holding for all $y\in\cY$ and $u\in\cU$ which are equivalent to
operator  equalities \eqref{4.26}.

\smallskip

By definitions \eqref{4.22} of $B^*$ and \eqref{4.8}, \eqref{4.9} of 
${\mathbb K}_j$ and ${\mathbb T}$,
\begin{align}
B^*\left({\bQ}(\zeta)^*\otimes{\mathbb K}_0(\cdot,\zeta)y\right)
&=\widetilde{\bf s}\left({\bQ}(\zeta)^*\otimes{\mathbb
K}_0(\cdot,\zeta)y\right)(0)=
\sum_{j=1}^d \overline{\zeta}_j\Psi_j(\zeta,0)^*y,
\label{4.32}\\
B^*{\mathbb T}(\cdot,\zeta)u&=\widetilde{\bf s}\left({\mathbb
T}(\cdot,\zeta)u\right)(0)=\sum_{j=1}^d \Phi_{jj}
(0,\zeta)u.\label{4.33}
\end{align}
Upon letting $z=0$ in \eqref{4.4} and \eqref{2.11} we get
\begin{equation}
S(\zeta)^*-S(0)^*=\sum_{j=1}^d \overline{\zeta}_j\Psi_j(\zeta,0)^*,
\quad I_{\cU}-S(0)^*S(\zeta)=
\sum_{j=1}^d \Phi_{jj}(0,\zeta)
\label{4.34}
\end{equation}
which being combined with \eqref{4.32} and 
\eqref{4.33} give \eqref{4.27} and \eqref{4.28}.
\end{proof}

Formulas \eqref{4.27}, \eqref{4.28} describing the action of the 
operator $B^*$ on elementary kernels of the subspace ${\mathcal D}$ 
defined in \eqref{4.13} were easily obtained from the general formula 
\eqref{4.22} for $B^*$. Although the operator $A^*$ is not defined in 
Definition \ref{D:4.1} 
on the whole space $\cH({\mathbb K})^d$, it turns out that 
its action on elementary kernels of ${\mathcal D}$ is completely 
determined by conditions \eqref{4.20} and \eqref{4.21}.
Formula \eqref{4.23} (which is equivalent to 
\eqref{4.20}) does half of the job; the next proposition takes care 
of the other half.

\begin{proposition}
Let ${\bf U}=\sbm {A & B \\ C & D}$ be a $\tcfm$ colligation 
associated with the Agler decomposition \eqref{4.3} of a given 
$S\in\cS_d(\cU,\, \cY)$ and let ${\mathbb T}$ be given 
by \eqref{4.9}. Then 
\begin{equation}
A^*{\mathbb T}(\cdot,\zeta)u=\sum_{j=1}^d \zeta_j
{\mathbb K}_j(\cdot,\zeta)u-{\mathbb K}_0(\cdot,0)S(\zeta)u\quad
(\zeta\in\B^d, \; y\in\cY, \; u\in\cU).
\label{4.35}
\end{equation}
\label{P:4.3}
\end{proposition}

\begin{proof} We have to show that formula \eqref{4.35} follows from 
conditions in Definition \ref{D:4.1}. To this end, we first verify
the 
equality
\begin{equation}
\left\|h_{\zeta,u}\right\|_{\cH({\mathbb K})}^2- 
\left\|Ah_{\zeta,u}\right\|_{\cH({\mathbb K})^d}^2=
\left\|Ch_{\zeta,u}\right\|_{\cU}^2,\quad\mbox{where}\quad
h_{\zeta,u}=\sum_{j=1}^d \zeta_j{\mathbb K}_j(\cdot,\zeta)u.
\label{4.36}  
\end{equation}
Indeed, it follows from the explicit formula \eqref{4.22} for $C$
that
\begin{equation}
Ch_{\zeta,u}={\bf s}\left(\sum_{j=1}^d 
\zeta_j{\mathbb K}_j(\cdot,\zeta)u\right)(0)
=\sum_{j=1}^{d}\zeta_j\Psi_j(0,\zeta) u =S(\zeta)u-S(0)u\label{4.37}
\end{equation}
where the last equality is a consequence of \eqref{4.4}. By the 
reproducing kernel property,
$$
\langle {\mathbb K}_\ell(\cdot,\zeta)u, \, {\mathbb 
K}_i(\cdot,\zeta)u\rangle_{\cH({\mathbb K})}=\langle 
\Phi_{i\ell}(\zeta,\zeta)u, \, u\rangle_{\cU}
$$
for $i,\ell=1,\ldots,d$, and therefore,
\begin{equation}
\left\|h_{\zeta,u}\right\|_{\cH({\mathbb K})}^2=
\left\|\sum_{j=1}^d
\zeta_j{\mathbb K}_j(\cdot,\zeta)u\right\|_{\cH({\mathbb K})}^2
=\sum_{i,\ell=1}^d
\overline{\zeta}_i\zeta_\ell\Phi_{i\ell}(\zeta,\zeta).
\label{4.38}
\end{equation}
Making use of \eqref{4.24} (which holds by Proposition \ref{P:4.2})
and 
of \eqref{4.30} we have
\begin{align}
\left\|Ah_{\zeta,u}\right\|_{\cH({\mathbb K})^d}^2&= 
\left\|{\mathbb T}(\cdot,\zeta)u-{\mathbb T}(\cdot,0)u
\right\|_{\cH({\mathbb K})^d}^2\notag\\
&=\sum_{j=1}^d\left\langle \left(\Phi_{jj}(\zeta,\zeta)-
\Phi_{jj}(\zeta,0)-\Phi_{jj}(0,\zeta)+\Phi_{jj}(0,0)\right)u, 
\, u\right\rangle_{\cU}. \label{4.39}
\end{align}
Observe that by \eqref{2.11},
\begin{align*}
&\sum_{j=1}^{d}\left(\Phi_{jj}(\zeta,\zeta)-
\Phi_{jj}(\zeta,0)-\Phi_{jj}(0,\zeta)+\Phi_{jj}(0,0)\right)
-\sum_{i,\ell=1}^{d}\overline{\zeta}_i\zeta_\ell
\Phi_{i\ell}(\zeta,\zeta)\notag\\
&=I_{\cU}-S(\zeta)^*S(\zeta)-
(I_{\cU}-S(\zeta)^*S(0))-
(I_{\cU}-S(0)^*S(\zeta))+I_{\cU}-S(0)^*S(0)\notag\\
&=-(S(\zeta)^*-S(0)^*)(S(\zeta)-S(0)).
\end{align*}
Subtracting \eqref{4.39} from \eqref{4.38} and taking into account
the 
last
identity we get
$$
\left\|h_{\zeta,u}\right\|^2-\left\|Ah_{\zeta,u}\right\|^2=
\left\|S(\zeta)u-S(0)u\right\|^2_{\cY}
$$
which proves \eqref{4.36}, due to \eqref{4.37}. Writing \eqref{4.36}
in 
the form
$$
\left\langle (I-A^*A-C^*C)h_{\zeta,u}, \, h_{\zeta,u}
\right\rangle_{\cH({\mathbb K})^p}=0
$$
and observing that the operator $I-A^*A-C^*C$ is
positive semidefinite (since $\bU$ is contractive by Definition
\ref{D:4.1}), we conclude that
\begin{equation}
(I-A^*A-C^*C)h_{\zeta,u}\equiv 
0\quad\mbox{for all}\quad
\zeta\in\B^d, \, u\in\cU.
\label{4.40} 
\end{equation}
Applying the operator $C^*$ to both parts of \eqref{4.37} we get
\begin{equation}
C^*Ch_{\zeta,u}={\mathbb K}_0(\cdot,0)\left(S(\zeta)-S(0)\right)u 
\label{4.41}
\end{equation}
by the explicit formula \eqref{4.25} for $C^*$. From the same formula 
and the formula \eqref{4.22} for $D$ we get
\begin{equation}
C^*Du=C^*S(0)^*u={\mathbb K}_0(\cdot,0)S(0)u.
\label{4.42}
\end{equation}
We next apply the operator $A^*$ to both parts of equality
\eqref{4.24} to 
get
$$
A^*Ah_{\zeta,u}= 
A^*{\mathbb T}(\cdot,\zeta)u-A^*{\mathbb T}(\cdot,0)u.
$$
Due to the second formula in \eqref{4.25}
(which holds by Proposition \ref{P:4.2}) the latter equality can be
written as
\begin{equation}
A^*{\mathbb T}(\cdot,\zeta)u=A^*Ah_{\zeta,u}+A^*Bu.   
\label{4.43}
\end{equation}
Since $\bU$ is contractive (by Definition \ref{D:4.1}) and 
since $B$ and $D$ satisfy the second equality in \eqref{4.26},
it then follows that $A^*B+C^*D=0$. Thus,
$$
A^*Bu=-C^*Du=-C^*S(0)^*u=-{\mathbb K}_0(\cdot,0)S(0)u.
$$
Taking the latter equality into account and making subsequent use of  
\eqref{4.40}--\eqref{4.42}  we then get from 
\eqref{4.43}
\begin{align*}
A^*{\mathbb T}(\cdot,\zeta)u&=(I-C^*C)h_{\zeta,u}-C^*Du\\ 
&=h_{\zeta,u}-
{\mathbb K}_0(\cdot,0)\left(S(\zeta)-S(0)\right)u 
-{\mathbb K}_0(\cdot,0)S(0)u \\
&=\sum_{j=1}^d \zeta_j   
{\mathbb K}_j(\cdot,\zeta)u-{\mathbb K}_0(\cdot,0)S(\zeta)u
\end{align*}
which completes the proof of \eqref{4.35}.
\end{proof}
\begin{remark}
{\rm Since any $\tcfm$ colligation is contractive, we have in
particular
that $AA^*+BB^*\le I$. Therefore, formulas \eqref{4.27}, \eqref{4.28}
and 
\eqref{4.35}, \eqref{4.23} defining the action of operators $B^*$ and 
$A^*$ on elementary kernels of the space ${\mathcal D}$ (see
\eqref{4.13})
can be extended by continuity to define these operators on the whole 
space ${\mathcal D}$.}
\label{R:4.3a}
\end{remark}

\begin{proposition}
Any $\tcfm$ colligation ${\bf U}=\sbm {A & B \\ C & D}$ associated
with a 
fixed Agler decomposition \eqref{4.3} of a given $S\in\cS_d(\cU, 
\, \cY)$ is weakly unitary and closely connected. Furthermore, 
\begin{equation}
S(z)=D+C(I-Z_{\cH({\mathbb K})}(z)A)^{-1}Z_{\cH({\mathbb K})}(z)B.
\label{4.44}
\end{equation}
\label{P:4.4}
\end{proposition}

\begin{proof} Let ${\bf U}=\sbm {A & B \\ C & D}$ be a 
$\tcfm$ colligation of $S$ associated with a fixed
Agler decomposition \eqref{4.3}. Then equalities  
\eqref{4.23}--\eqref{4.25} hold by Proposition 
\ref{P:4.2}. Upon representing the left hand side expressions 
in \eqref{4.23}, \eqref{4.25} as 
$$
A^*Z_{\cH({\mathbb K})}(\zeta)^*{\mathbb K}_0(\cdot, 
\zeta)y\quad\mbox{and}\quad AZ_{\cH({\mathbb
K})}(\zeta){\mathbb T}(\cdot,\zeta)u
$$ 
respectively and replacing $\zeta$ by $z$, we then solve the 
system \eqref{4.23}--\eqref{4.25} for ${\mathbb
T}(\cdot,z)y$ and $\widetilde{\mathbb T}(\cdot,z)u$ as follows:
\begin{align}
{\mathbb K}_0(\cdot,z)y&=
(I-A^*Z_{\cH({\mathbb K})}(z)^*)^{-1}{\mathbb K}_0(\cdot,0)y
=(I-A^*Z_{\cH({\mathbb K})}(z)^*)^{-1}C^*y,\label{4.45}\\
{\mathbb T}(\cdot,z)u&=
(I-AZ_{\cH({\mathbb K})}(z))^{-1}{\mathbb T}(\cdot,0)u
=(I-AZ_{\cH({\mathbb K})}(z))^{-1}Bu.\label{4.46}
\end{align}
From \eqref{4.45} and \eqref{4.27} we conclude that equalities
\begin{align}
&(D^*+B^*Z_{\cH({\mathbb K})}(z)^*(I-A^*Z_{\cH({\mathbb 
K})}(z)^*)^{-1}C^*)y\label{4.47}\\
&\quad=S(0)^*y+B^*Z_{\cH({\mathbb K})}(z)^*{\mathbb K}_0(\cdot,z)y\notag\\
&\quad=S(0)^*y+S(z)^*y-S(0)^*y=S(z)^*y\notag
\end{align}
hold for every $z\in\B^d$ and $y\in\cY$, which proves representation 
\eqref{4.44}. Furthermore, in view of \eqref{4.8} and \eqref{4.9},
\begin{align*}
\cH^{\mathcal O}_{C,A}&:=\bigvee\left\{
(I-A^*Z_{\cH({\mathbb K})}(z)^*)^{-1}C^*y: \; z\in \B^d, \; 
y\in\cY\right\}\notag\\
&=\bigvee\left\{{\mathbb K}_0(\cdot,z)y: 
\, z\in \B^d, \; y\in\cY\right\},\\
\cH^{\mathcal C}_{A,B}&:=\bigvee\left\{{\mathcal
I}_{j}^{*}(I-AZ_{\cH({\mathbb K})}(z))^{-1}Bu: \; z\in \B^d, \; u\in\cU,
\; 
j=1,\ldots,d\right\}\notag\\
&=\bigvee\left\{{\mathcal I}_{j}^{*}{\mathbb T}(\cdot,z)u:
\, z\in \B^d, \; u\in\cU, \;  j=1,\ldots,d\right\}\notag\\
&=\bigvee\left\{{\mathbb K}_j(\cdot,z)u_j
: \; z\in \B^d, \; u_j\in\cU, \; j=1,\ldots,d\right\},
\end{align*}
and therefore, 
\begin{align*}
\cH^{\mathcal O}_{C,A}\bigvee \cH^{\mathcal C}_{A,B}&=
\bigvee\left\{{\mathbb K}_0(\cdot,z)y, \; {\mathbb 
K}_j(\cdot,z)u_j: \; z\in \B^d, \; y\in\cY, \,  
u_j\in\cU, \, j=1,\ldots,d\right\}\\
&=\bigvee\left\{{\mathbb K}(\cdot,z)\begin{bmatrix}y \\
{\bf u}\end{bmatrix}: \; z\in \B^d, \; \begin{bmatrix}y \\
{\bf u}\end{bmatrix}\in\cY\oplus\cU^d\right\}=\cH({\mathbb K})
\end{align*}
where the last equality follows by  the very construction of the 
reproducing kernel Hilbert space. The colligation ${\bf U}=\sbm {A &
B \\ 
C & D}$ is closely connected by Definition \ref{D:1.4}.
To show that $\bU$ is weakly unitary, we let $u=u'=0$, $y=y'$ and 
$z=\zeta$ in \eqref{4.10} to get 
$$
\left\| \left[\begin{array}{c}
Z_{\cH({\mathbb K})}(\zeta)^*{\mathbb K}_0(\cdot,\zeta)y\\
y\end{array}\right]\right\|=\left\|\left[\begin{array}{c}
{\mathbb K}_0(\cdot,\zeta)y \\ S(\zeta)^*y\end{array}\right]
\right\|
$$
which on account of \eqref{4.45} can be written as
\begin{equation}
\left\|\left[\begin{array}{c}Z_{\cH({\mathbb 
K})}(\zeta)^*(I-A^*Z_{\cH({\mathbb 
K})}(\zeta)^*)^{-1}C^*y\\
y\end{array}\right]\right\|=\left\|\left[\begin{array}{c}
(I-A^*Z_{\cH({\mathbb K})}(\zeta)^*)^{-1}C^*y\\ 
S(\zeta)^*y\end{array}\right]\right\|.
\label{4.48}
\end{equation}
Since
$$
 \begin{bmatrix}A^* & C^* \\ B^* &
D^*\end{bmatrix} \begin{bmatrix} Z_{\cH({\mathbb K})}(\zeta)^*
(I-A^*Z_{\cH({\mathbb K})}(\zeta)^*)^{-1}C^*y \\ y\end{bmatrix}  
 =\begin{bmatrix}
(I-A^*Z_{\cH({\mathbb K})}(\zeta)^*)^{-1}C^*y \\ 
S(\zeta)^*y  \end{bmatrix}
$$
(the top components in the latter formula are equal automatically
whereas 
the bottom components are equal due to \eqref{4.47}),
equality \eqref{4.48} tells us that $\bU$ is weakly 
coisometric by Definition \ref{D:1.4}. Similarly letting
$u=u'$ and $y=y'=0$ in \eqref{4.10} we get
$$
\left\| \begin{bmatrix}
{\mathbb T}(\cdot,\zeta)u \\ S(\zeta)u \end{bmatrix}\right\|=
\left\|\begin{bmatrix} Z_{\cH({\mathbb K})}(\zeta){\mathbb
T}(\cdot,\zeta)u
\\ u \end{bmatrix}\right\|
$$
which in view of \eqref{4.46} can be written as
$$
\left\|\left[\begin{array}{c}(I-AZ_{\cH({\mathbb K})}(\zeta))^{-1}Bu\\
S(\zeta)u\end{array}\right]\right\|=\left\|\left[\begin{array}{c}
Z_{\cH({\mathbb K})}(\zeta)(I-AZ_{\cH({\mathbb K})}(\zeta))^{-1}Bu\\ 
u\end{array}\right]\right\|
$$
and since 
$$
\begin{bmatrix}A & B \\ C & D\end{bmatrix}\left[\begin{array}{c}
Z_{\cH({\mathbb K})}(\zeta)(I-AZ_{\cH({\mathbb K})}(\zeta))^{-1}Bu\\ 
u\end{array}\right]=\left[\begin{array}{c}
(I-AZ_{\cH({\mathbb K})}(\zeta))^{-1}Bu\\ S(\zeta)u\end{array}\right]
$$
(again, the top components are equal automatically 
and the bottom components are equal due to \eqref{4.44}),
the colligation $\bU$ is weakly isometric by Definition \ref{D:1.4}.
\end{proof}

\medskip

Proposition \ref{P:4.4} establishes common features of $\tcfm$ 
colligations leaving the question about the existence of at least one 
such colligation open. 

\begin{lemma}
Given an Agler decomposition ${\mathbb K}$ for a function
$S\in{\mathcal {S}}_d(\cU, \cY)$, let
$V$ be the isometric operator associated with this decomposition as
in \eqref{4.11}. A block-operator matrix $\bU=\sbm {A & B \\ C & D}$
of the form \eqref{4.19}
is a $\tcfm$ colligation associated with  ${\mathbb K}$ if and only if
\begin{equation}
\|\bU^*\|\le 1,\quad \bU^*\vert_{{\mathcal 
D}\oplus\cY}=V\quad\mbox{and}\quad
B^*\vert_{{\mathcal D}^\perp}=0,
\label{4.49}
\end{equation}
that is, $\bU^*$  is a contractive extension of $V$ from ${\mathcal
D}\oplus \cY$ to all of $\cH({\mathbb K})^d\oplus\cY$ subject to the
condition $B^*\vert_{{\mathcal D}^\perp}=0$.
\label{L:1}
\end{lemma}
\begin{proof} Let us write the isometry $V$ from \eqref{4.11} in the
form
\begin{equation}
V=\begin{bmatrix}A_V & B_V \\ C_V & D_V\end{bmatrix}: \;
\begin{bmatrix}
{\mathcal D} \\ \cY\end{bmatrix}\to \begin{bmatrix}
{\mathcal R} \\ \cU\end{bmatrix}.
\label{4.50}  
\end{equation}
Then we get from \eqref{4.11} the following relations for the block 
entries $A_V$, $B_V$,  $C_V$, $D_V$:
\begin{align}
A_V\left({\bQ}(\zeta)^*\otimes{\mathbb
K}_0(\cdot,\zeta)y\right)+B_Vy&=
{\mathbb K}_{0}(\cdot,\zeta)y,\label{4.51}\\
A_V{\mathbb T}(\cdot,\zeta)u
+B_VS(\zeta)u&=\sum_{j=1}^d\zeta_j{\mathbb 
K}_j(\cdot,\zeta)u,\label{4.52}\\
C_V\left({\bQ}(\zeta)^*\otimes{\mathbb 
K}_0(\cdot,\zeta)y\right)+D_Vy&=S(\zeta)^*y,
\label{4.53}\\
C_V{\mathbb T}(\cdot,\zeta)u +D_VS(\bar{\zeta})u&=u.\label{4.54}
\end{align}
Indeed, equalities \eqref{4.51} and \eqref{4.52} are obtained upon 
equating the top components in \eqref{4.11} for the respective 
special cases $u=0$ and $y=0$.
Equalities \eqref{4.53} and \eqref{4.54} are obtained similarly upon
equating the bottom components in \eqref{4.11}. Letting $\zeta=0$ in
\eqref{4.51} and \eqref{4.53} gives
\begin{equation}
B_Vy={\mathbb K}_0(\cdot,0)y\quad\mbox{and}\quad D_Vy=S(0)^*y.
\label{4.55}
\end{equation}
Substituting the first and the second formula in \eqref{4.55}
respectively into \eqref{4.51}, \eqref{4.52} and into \eqref{4.53} 
and \eqref{4.54} results in equalities
\begin{align}
A_V\colon & {\bQ}(\zeta)^*\otimes{\mathbb K}_0(\cdot,\zeta)y\to
{\mathbb K}_0(\cdot,\zeta)y-{\mathbb K}_0(\cdot,0)y,\label{4.56}\\
A_V\colon & {\mathbb T}(\cdot,\zeta)u\to \sum_{j=1}^d\zeta_j
{\mathbb K}_j(\cdot,\zeta)u
-{\mathbb K}_0(\cdot,0)S(\zeta)u,\label{4.57}\\
C_V\colon & {\bQ}(\zeta)^*\otimes{\mathbb K}_0(\cdot,\zeta)y
\to S(\zeta)^*y-S(0)^*y,\label{4.58}\\
C_V\colon & {\mathbb T}(\cdot,\zeta)u\to u-S(0)^*S(\bar{\zeta})u
\label{4.59}
\end{align}
holding for all $\zeta\in\B^d$, $u\in\cU$ and $y\in\cY$ and 
completely defining the operators $A_V$ and $C_V$ on the whole space 
${\mathcal D}$.

\smallskip

Let $\bU=\sbm {A & B \\ C & D}$ be a $\tcfm$ colligation associated
with 
${\mathbb K}$. Then $\bU$ is contractive by definition and 
relations \eqref{4.23}--\eqref{4.25} and \eqref{4.35} hold 
by Propositions \ref{P:4.2} and \ref{P:4.3}. Comparing \eqref{4.23}
and 
\eqref{4.35} with \eqref{4.56}, \eqref{4.57} we see that
$A^*|_{\mathcal 
D}=A_V$. Comparing \eqref{4.27}, \eqref{4.28} with \eqref{4.58}, 
\eqref{4.59} we conclude that $B^*|_{\mathcal D}=C_V$. Also, it 
follows from \eqref{4.25} and \eqref{4.55} that $C^*=B_V$ and
$D^*=D_V$. 
Finally, the second of formulas \eqref{4.22} combined with the 
characterization  \eqref{4.15} of $\cD^{\perp}$ enables us to see 
that $B^*f=\widetilde{\bf s}f=0$ for every $f\in{\mathcal D}^\perp$.
The last equality in \eqref{4.49} now follows.

\smallskip

Conversely, let us assume that a colligation $\bU=\sbm {A & B \\ C &
D}$  
meets all the conditions in \eqref{4.49}. From the second relation in 
\eqref{4.49} we conclude the equalities \eqref{4.55}--\eqref{4.59}
hold 
with operators 
$A_V$, $B_V$, $C_V$ and $D_V$ replaced by $A^*$, $C^*$. $B^*$ and
$D^*$  
respectively. In other words, we conclude from \eqref{4.55} that
$C^*$ 
and $D^*$ are defined exactly as in \eqref{4.25} which means (by 
Proposition \ref{P:4.2}) that they are already of the requisite form. 
Equalities \eqref{4.58}, \eqref{4.59} tell us that the operator $B^*$ 
satisfies formulas  \eqref{4.27}, \eqref{4.28}. As we have seen in
the proof of Proposition \ref{P:4.3}, these formulas agree with the
second 
formula in \eqref{4.22} and then define $B^*$ on the whole space
$\cH({\mathbb K})^d$. From the third condition in \eqref{4.49} 
we now conclude that $B^*$ is defined 
by formula  \eqref{4.22} on the whole space $\cH({\mathbb K})^d$ and 
therefore, $B$ is also of the requisite form. The formula
\eqref{4.56} 
(with $A^*$ instead of $A_V$) leads us to \eqref{4.23} which means
that 
$A$ solves the Gleason problem \eqref{4.20}. 

\smallskip

To complete the proof, it remains to show that $A^*$ solves the dual 
Gleason problem \eqref{4.21} or equivalently, that \eqref{4.24}
holds. 
Rather than \eqref{4.24}, what we know is the equality \eqref{4.52} 
(with $A^*$ and $C^*$ instead of $A_V$ and $B_V$ respectively):
\begin{equation}
A^*{\mathbb T}(\cdot,\zeta)u=\sum_{j=1}^d\zeta_j
{\mathbb K}_j(\cdot,\zeta)u-C^*S(\zeta)u. \label{4.60}
\end{equation}
We use \eqref{4.60} to show that equality  
\begin{equation}
\left\|{\mathbb T}(\cdot,\zeta)u\right\|_{\cH({\mathbb K})^d}^2
-\left\|A^*{\mathbb T}(\cdot,\zeta)u
\right\|_{\cH({\mathbb K})}^2=\left\|B^*{\mathbb T}(\cdot,\zeta)u
\right\|_{\cU}^2
\label{4.61}  
\end{equation}
holds for every $\zeta\in\B^d$ and $u\in\cU$. Indeed, it follows from 
\eqref{4.60} that
\begin{align}
\left\|{\mathbb T}(\cdot,\zeta)u\right\|^2-\left\|A^*{\mathbb 
T}(\cdot,\zeta)u\right\|^2=& \|{\mathbb 
T}(\cdot,\zeta)u\|^2-\|
\sum_{j=1}^d\zeta_j {\mathbb K}_j(\cdot,\zeta)u
-C^*S(\zeta)u\|^2\notag \\
=&\left\|{\mathbb T}(\cdot,\zeta)u\right\|^2
-\|\sum_{j=1}^d\zeta_j{\mathbb K}_j(\cdot,\zeta)u\|^2-
\left\|C^*S(\zeta)u\right\|^2\notag \\
&-2\Re\left\langle C\left(\sum_{j=1}^d\zeta_j{\mathbb 
K}_j(\cdot,\zeta)u\right), \, 
S(\zeta)u\right\rangle.
\label{4.62} 
\end{align}
We next express all the terms on the right of \eqref{4.62} in terms
of 
the function $S$: 
\begin{align*}
\left\|{\mathbb T}(\cdot,\zeta)u\right\|^2-\|
\sum_{j=1}^d\zeta_j{\mathbb K}_j(\cdot,\zeta)u\|^2&
=\left\langle (I_{\cU}-S(\zeta)^*S(\zeta))u, \,
u\right\rangle,\\
\left\langle C\left(\sum_{j=1}^d\zeta_j{\mathbb
K}_j(\cdot,\zeta)u\right), 
\, S(\zeta)u\right\rangle&
=\left\langle S(\zeta)^*(S(\zeta)-S(0))u, \,u 
\right\rangle,\\ \left\|C^*S(\zeta)u\right\|^2&=\|S(\zeta)u\|^2-
\|S(0)^*S(\zeta)u\|^2.
\end{align*}
We mention that the first equality follows from \eqref{4.30},
\eqref{4.38} 
and  \eqref{2.11}; the second equality is a consequence of 
\eqref{4.37}; the third equality is obtained upon letting
$y=S(\zeta)u$ in 
\eqref{4.31}.
We now substitute the three last equalities into \eqref{4.62} to get 
\begin{equation}
\left\|{\mathbb T}(\cdot,\zeta)u\right\|_{\cH({\mathbb K})^d}^2
-\left\|A^*{\mathbb T}(\cdot,\zeta)u\right\|_{\cH({\mathbb K})}^2
=\left\langle R(\zeta)u,\, u\right\rangle_{\cU}
\label{4.63}
\end{equation}
where 
\begin{eqnarray*}
R(\zeta)&=&I_{\cU}-S(\zeta)^*S(\zeta)+
S(\zeta)^*\left(S(\zeta)-S(0)\right)\\
&&+\left(S(\zeta)^*-S(0)^*\right)S(\zeta)-S(\zeta)^*S(\zeta)+
S(\zeta)^*S(0)S(0)^* S(\zeta)\\
&=& I_{\cU}-S(\zeta)^*S(0)-S(0)^*S(\zeta)+
S(\zeta)^*S(0)S(0)^*S(\zeta)\\
&=&\left(I_\cU-S(\zeta)^*S(0)\right)\left(I_\cU-S(0)^*S(\zeta)\right).
\end{eqnarray*}
By \eqref{4.28} we have
\begin{equation}
B^*{\mathbb T}(\cdot,\zeta)u=u-S(0)^*S(\zeta)u
\label{4.64}
\end{equation}
and therefore
$$
\left\|B^*{\mathbb T}(\cdot,\zeta)u\right\|_{\cU}^2=
\left\|u-S(0)^*S(\zeta)u\right\|^2=
\left\langle R(\zeta)u,\, u\right\rangle_{\cU},
$$
which together with \eqref{4.63} completes the proof of \eqref{4.61}.
Writing \eqref{4.61} as
$$
\langle (I-AA^*-BB^*){\mathbb T}(\cdot,\zeta)u, \, 
{\mathbb T}(\cdot,\zeta)u\rangle=0
$$
and observing that the operator $I-AA^*-BB^*$ is positive
semidefinite 
(since ${\bf U}=\sbm {A&B \\ C & D}$ is a contraction), we conclude
that 
\begin{equation}
(I-AA^*-BB^*){\mathbb T}(\cdot,\zeta)u=0\quad\mbox{for all} 
\quad \zeta\in\B^d, \, u\in\cU.
\label{4.65}  
\end{equation}
Since the operators $C$ and $D$ satisfy the first equality
\eqref{4.26}
and since  ${\bf U}=\sbm {A&B \\ C & D}$ is a contraction, 
necessarily we have 
$AC^*+BD^*=0$. We now combine this latter equality with 
\eqref{4.64} and formula \eqref{4.25} for $D^*$ to get 
\begin{align}
{\mathbb T}(\cdot,0)u=Bu&=B(B^*{\mathbb T}(\cdot,\zeta)u+
S(0)^*S(\zeta)u)\notag\\ &=
BB^*{\mathbb T}(\cdot,\zeta)u+BD^*S(\zeta)u\notag\\
&=BB^*{\mathbb T}(\cdot,\zeta)u-AC^*S(\zeta)u.
\label{4.66}
\end{align}
We now apply the operator $A$ to both parts of \eqref{4.60},
$$
AA^*{\mathbb T}(\cdot,\zeta)u=A\sum_{j=1}^d\zeta_j{\mathbb 
K}_j(\cdot,\zeta)u-AC^*S(\zeta)u
$$
and combine the obtained identity with \eqref{4.65} and \eqref{4.66}:
\begin{align*}
A\left(\sum_{j=1}^d\zeta_j{\mathbb
K}_j(\cdot,\zeta)u\right)&=AA^*{\mathbb
T}(\cdot,\zeta)u+AC^*S(\zeta)u\\
&={\mathbb T}(\cdot,\zeta)u-BB^*{\mathbb
T}(\cdot,\zeta)u-BD^*S(\zeta)u\\
&={\mathbb T}(\cdot,\zeta)u-{\mathbb T}(\cdot,0)u. 
\end{align*}
This completes the proof of  \eqref{4.24}.
\end{proof}

As a consequence of Lemma \ref{L:1} we get a description of all
$\tcfm$ colligations associated with a given Agler
decomposition of a function $S\in\cS_d(\cU,\cY)$.

\begin{lemma}
\label{L:2.5}                 
Let ${\mathbb K}$ be a fixed Agler decomposition of a
function  $S\in\cS_d(\cU, \cY)$. Let $V$ be the 
associated isometry defined in \eqref{4.11} with the defect spaces 
$\cD^{\perp}$ and ${\mathcal R}^\perp$ defined in \eqref{4.15}, 
\eqref{4.16}. Then all $\tcfm$ colligations associated with ${\mathbb
K}$
are of the form 
\begin{equation}
\bU^*=\begin{bmatrix} X & 0  \\ 0 & V\end{bmatrix}\colon
\begin{bmatrix}\cD^\perp \\ \cD\oplus\cY\end{bmatrix}\to
\begin{bmatrix}\cR^\perp \\ \cR\oplus\cU\end{bmatrix} 
\label{4.67}
\end{equation}
where we have identified $\begin{bmatrix}\cH({\mathbb K})^d\\
\cY\end{bmatrix}$ with $\begin{bmatrix} \cD^{\perp} \\ \cD\oplus\cY
\end{bmatrix}$ and $\begin{bmatrix}\cH({\mathbb K})\\
\cU\end{bmatrix}$ with $\begin{bmatrix} \cR^{\perp} \\ \cR\oplus\cU
\end{bmatrix}$ and where $X$ is an arbitrary contraction from 
$\cD^\perp$ into $\cR^\perp$. The colligation $\bU$ is isometric
(coisometric, unitary) if and only if $X$ is coisometric (isometric, 
unitary).
\end{lemma}

For the proof, it is enough to recall that $V$ is unitary  as an
operator  from $\cD_V=\cD\oplus \cY$ onto $\cR_V=\cR\oplus\cU$ and then to
refer to 
Lemma \ref{L:1}. The meaning of description \eqref{4.67} is clear: the 
operators $B^*$, $C^*$, $D^*$ and the restriction of $A^*$ to the 
subspace $\cD$ in the operator colligation $\bU^*$ are prescribed. 
The objective is to guarantee $\bU^*$ be contractive by suitably 
defining $A^*$ on $\cD^\perp$. Lemma \ref{L:2.5} states that 
$X=A^*\vert_{{\cD}^\perp}$ must be a contraction with range contained
in $\cR^\perp$.

\medskip

We now are ready to formulate the main result of this section.

\begin{theorem}  
\label{T:4.7}
Let $S$ be a function in $\cS_{d}(\cU, \cY)$ with given Agler 
decomposition ${\mathbb K}$. Then
\begin{enumerate}
\item There exists a $\tcfm$ colligation $\bU=\sbm{A & B \\ C & D}$ 
associated with ${\mathbb K}$.
\item Every $\tcfm$ colligation $\bU$ associated with
${\mathbb K}$ is weakly unitary and closely connected and
furthermore, $S(z) = D+C(I-Z_{\cH({\mathbb K})}(z)A)^{-1}
Z_{\cH({\mathbb K})}(z) B$.
\item Any weakly unitary closely connected colligation
$\widetilde{\bU}$
of the form \eqref{1.16} with the  characteristic  function equal
$S$ is unitarily equivalent to a $\tcfm$ colligation $\bU$ associated 
with associated Agler decomposition ${\mathbb K}_{\widetilde \bU}$
for $S$
given by \eqref{4.2} with $\Psi_{k}$ and $\Phi_{ij}$ given as in 
\eqref{4.3''} and \eqref{4.3'''}:
\begin{align}
\Psi_k(z,\zeta) &=\tC(I-Z_\cX(z)\widetilde{A})^{-1}{\mathcal I}_k^{*}
(I-\widetilde{A}Z_\cX(\zeta))^{-1}\tB    \label{4.74}    \\
\Phi_{i j}(z,\zeta)&=\widetilde{B}^*(I-Z_{\cX}(z)^*\widetilde{A}^*)^{-1}
{\mathcal I}_{i}{\mathcal I}_j^{*}
(I-\widetilde{A}Z_\cX(\zeta))^{-1}\tB\label{4.75} 
\end{align}
where the inclusion operators ${\mathcal I}_{j}$ are as in
\eqref{1.13}.
\end{enumerate}
\end{theorem}

\begin{proof} Part (1) is contained in Lemma \ref{L:2.5}. Part (2)
was 
proved in Proposition \ref{P:4.4}. To prove part (3) we assume that 
$\widetilde{\bU}=\sbm{\tA & \tB \\ \tC & D}\colon \;
\sbm{\cX \\ \cU}\to \sbm{\cX^d \\ \cY}$ is a closely connected 
weakly unitary colligation with the state space $\cX$ and such that 
\begin{equation} 
S(z)=D+\widetilde{C}(I-Z_{\cX}(z)\widetilde{A})^{-1}Z_{\cX}(z)\widetilde{B}.
\label{4.68}
\end{equation}
The proof of unitary equivalence of $\widetilde{\bU}$ to a $\tcfm$ 
colligation for $S$ associated with the Agler decomposition as in 
\eqref{4.74}, \eqref{4.75} will be broken into three steps below. Let
${\mathbb 
G}(z)$ be the  operator-valued function
\begin{equation}
{\mathbb G}(z):=\left[\begin{array}{c}
\widetilde{C} (I-Z_{\cX}(z)\widetilde{A})^{-1}\\
\widetilde{B}^*(I-Z_{\cX}(z)^*\widetilde{A}^*)^{-1}{\mathcal 
I}_{1} \\ \vdots \\ 
\widetilde{B}^*(I-Z_{\cX}(z)^*\widetilde{A}^*)^{-1}{\mathcal
I}_{d} \end{array}\right]: \; \B^d\to \cL(\cX, \cY\oplus\cU^d),
\label{4.69}
\end{equation}
with the operators ${\mathcal I}_{j}$  as in \eqref{1.13}.

\medskip

{\bf Step 1:} {\em The Agler decomposition \eqref{4.3} holds for the 
kernel ${\mathbb K}$ given by}
\begin{equation}
{\mathbb K}(z,\zeta)={\mathbb G}(z){\mathbb G}(\zeta)^*
\label{4.70}  
\end{equation}
{\bf Proof of Step 1:} It follows by straightforward calculations
(see e.g., \cite{BTV}) that for the  characteristic  function $S$
\eqref{4.68} 
of the colligation $\widetilde{\bU}=\sbm{\tA & \tB \\ \tC & D}$,
\begin{align*}
I_{\cY}-S(z)S(\zeta)^*=&(1-\langle z,\zeta\rangle)\tC (I-Z_\cX(z)
\widetilde{A})^{-1}(I-\widetilde{A}^*Z_\cX(\zeta)^*)^{-1}\tC^*\\
&+F(z)\left(I-\widetilde{\bU}\widetilde{\bU}^*\right)F(\zeta)^*
\end{align*}
where $F(z)=\begin{bmatrix}\tC (I-Z_\cX(z)\widetilde{A})^{-1}Z_\cX(z)
& I\end{bmatrix}$, and 
\begin{align*}
I-S(z)^*S(\zeta)=&\tB^*(I-Z_\cX(z)^* \widetilde{A}^*)^{-1}
\left(I-Z_\cX(z)^*Z_\cX(\zeta)\right)(I-\widetilde{A}
Z_\cX(\zeta))^{-1}\tB\\ 
&+\widetilde{F}(z)\left(I-\widetilde{\bU}^*\widetilde{\bU}\right)
\widetilde{F}(\zeta)^*
\end{align*}
where 
$\widetilde{F}(z)=\begin{bmatrix}\tB^*(I-Z_\cX(z)^*\widetilde{A}^*)^{-1}
Z_\cX(z)^* & I\end{bmatrix}$, from which it is clear that
weak-coisometric 
and  weak-isometric properties of $\widetilde{\bU}$ (see Definition 
\ref{D:1.4}) are exactly what is needed for the respective identities 
\begin{align}
K_S(z,\zeta)&=\tC (I-Z_\cX(z)
\widetilde{A})^{-1}(I-\widetilde{A}^*Z_\cX(\zeta)^*)^{-1}\tC^*,
\label{4.71}\\
I-S(z)^*S(\zeta)&=\tB^*(I-Z_\cX(z)^* \widetilde{A}^*)^{-1}
\left(I-Z_\cX(z)^*Z_\cX(\zeta)\right)(I-\widetilde{A}Z_\cX(\zeta))^{-1}\tB.
\label{4.72}
\end{align}
Since  $\widetilde{\bU}$ is weakly unitary by assumption, the two
latter 
identities hold. Also we  observe that for $S$ of the form
\eqref{4.68}, 
\begin{align}
S(z)-S(\zeta)&=\tC (I-Z_\cX(z)
\widetilde{A})^{-1}Z_\cX(z)\tB-\tC Z_\cX(\zeta)(I-
\widetilde{A}Z_\cX(\zeta))^{-1}\tB\notag\\
&=\tC (I-Z_\cX(z)
\widetilde{A})^{-1}\left(Z_\cX(z)-Z_\cX(\zeta)\right)(I-        
\widetilde{A}Z_\cX(\zeta))^{-1}\tB.\label{4.73}
\end{align}
We now conclude from \eqref{4.69}, \eqref{4.70} and 
\eqref{4.71}--\eqref{4.73} that 
the kernel ${\mathbb K}$ indeed extends $K_S$ and is of the form 
\eqref{4.2} with $\Psi_{k}(z, \zeta)$ and $\Phi_{i j}(z, 
\zeta)$ given by \eqref{4.74} and \eqref{4.75}
for $k,i,j=1,\ldots,d$. It follows from \eqref{4.74}, \eqref{1.13}
and
\eqref{4.73} that
\begin{align*}
\sum_{k=1}^{d}(z_k-\zeta_k)\Psi_k(z,\zeta)&=\tC(I-Z_\cX(z)\widetilde{A})^{-1}
\left(\sum_{k=1}^{d}(z_k-\zeta_k){\mathcal I}_k^{*}\right)
(I-\widetilde{A}Z_\cX(\zeta))^{-1}\tB\\
&=\tC (I-Z_\cX(z)
\widetilde{A})^{-1}\left(Z_\cX(z)-Z_\cX(\zeta)\right)(I-
\widetilde{A}Z_\cX(\zeta))^{-1}\tB \\
&=S(z)-S(\zeta)
\end{align*}
so that equality \eqref{4.4} holds. Equality \eqref{2.11} follows in
much 
the same way from \eqref{4.75} and \eqref{4.72}. Thus, the identity 
\eqref{4.3} holds which completes the proof.

\medskip

{\bf Step 2:} {\em The linear map $U: \, \cX\to \cH({\mathbb K})$
defined
by the formula
\begin{equation}
U: \; x \mapsto  {\mathbb G}(z)x
\label{4.76}
\end{equation}
is unitary}.

\medskip 

{\bf Proof of Step 2:} Due to factorization
\eqref{4.70}, the reproducing kernel Hilbert space $\cH({\mathbb K})$
can be characterized as the range space
$$
\cH({\mathbb K})=\left\{f(z)={\mathbb G}(z)x: \; x\in\cX\right\}
$$
with the lifted norm $\|{\mathbb G}x\|_{\cH({\mathbb
K})}=\|(I-\pi)x\|_{\cX}$
where $\pi$ is the orthogonal projection onto the subspace
$\cX^\circ=\left\{x\in\cX: \, {\mathbb G}x\equiv 0\right\}$.
For every vector $x\in\cX^\circ$ we have by \eqref{4.69},
$$
\widetilde{C} (I-Z_\cX(z)\widetilde{A})^{-1}x \equiv 0
\quad\mbox{and}\quad
\widetilde{B}^*(I-Z_\cX(z)^*\widetilde{A}^*)^{-1}{\mathcal I}_{j}^*x 
\equiv 0
$$
for all $j=1,\ldots,d$. Then $x$ is orthogonal to
the spaces $\cH^{\mathcal O}_{\tC,\tA}$ and $\cH^{\mathcal
C}_{\tA,\tB}$
(see Definition \ref{D:1.4}) and since the colligation
$\widetilde{\bU}$
is closely connected, it follows that $x=0$. Thus, $\cX^\circ$ is
trivial
and $\|{\mathbb G}x\|_{\cH({\mathbb K})}=\|x\|_{\cX}$ which means that
the operator $U: \, x\to {\mathbb G}(z)x$ is a unitary operator 
from $\cX$ to $\cH({\mathbb K})$.

\medskip

{\bf Step 3:} {\em Define the operators $A: \, \cH({\mathbb
K})\to \cH({\mathbb K})^d$, $B: \,  \cU\to\cH({\mathbb K})^d$
and $C: \, \cH({\mathbb K})\to\cY$ by
\begin{equation}
A U=\left(\oplus_{j=1}^{d} U\right)
\widetilde{A},\quad B= \left(\oplus_{j=1}^{d}
U\right)\widetilde{B}\quad\mbox{and}\quad CU=\widetilde{C}
\label{4.77}
\end{equation}
where  $U: \, \cX\to \cH({\mathbb K})$ is defined in \eqref{4.76}.
The colligation $\bU=\sbm{A&B\\ C&D}$ is a $\tcfm$ colligation
associated 
with the Agler decomposition ${\mathbb K}$ for $S$}.

\medskip

{\bf Proof of Step 3:} We first observe that the colligation $\bU=
\sbm {A&B \\ C & D}$ is a contraction since it is unitarily
equivalent 
to a weakly unitary colligation $\widetilde{\bU}$. It remains to show
that 
$A$ solves the Gleason problems \eqref{4.20}, \eqref{4.21} and that 
$C$ and $B^*$ are of the form \eqref{4.22}.

\smallskip

Take the generic element $f$ of $\cH({\mathbb K})$ in the form
$$
f(z)={\mathbb G}(z)x, \quad x\in\cX
$$
so that $f=Ux$ by \eqref{4.76}, or equivalently, $x=U^*f$, since $U$
is 
unitary.
By definitions \eqref{4.17} and \eqref{4.69} we have
\begin{equation}
({\bf s}f)(z)=\tC(I-Z_\cX(z)\tA)^{-1}x.
\label{4.78}        
\end{equation}
Upon evaluating the latter equality at $z=0$ we get for the operator
$C$ 
from \eqref{4.77}
$$
Cf=\tC U^*f=\tC x=({\bf s}f)(0)
$$
so that the formula \eqref{4.22} for $C$ holds. We also have from 
\eqref{4.78}
\begin{equation}
({\bf s}f)(z)-({\bf s}f)(0)=\tC(I-Z_\cX(z)\tA)^{-1}x-
\tC x=\tC(I-Z_\cX(z)\tA)^{-1}Z_\cX(z)\tA x.
\label{4.79}
\end{equation}
On the other hand, for the operator $A$ defined in \eqref{4.77}, we
have
$$
Z_{\cH({\mathbb K})}(z)Af=Z_{\cH({\mathbb K})}(z) AUx
=Z_{\cH({\mathbb K})}(\oplus_{j=1}^d U)\tA x=UZ_\cX(z)\tA x
$$
and therefore, by formula \eqref{4.78} applied to $Z_\cX(z)\tA x$ 
rather than to $x$ we get
$$
{\bf s}(Z_{\cH({\mathbb K})}(z)Af)(z)=\tC(I-Z_\cX(z)\tA)^{-1}
Z_\cX(z)\tA{\bf x}
$$
which together with \eqref{4.79} implies \eqref{4.20}.

\smallskip

We now take the generic element $g$ of $\cH({\mathbb K})^d$ in the
form
$$
g(z)=\bigoplus_{j=1}^d {\mathbb G}(z)x_j
\quad\mbox{and let}\quad {\bf x}:=\bigoplus_{j=1}^d  x_j\in\cX^d,
$$
so that ${\bf x}=(\oplus_{j=1}^d U^*)g$. By definitions \eqref{4.17}
and 
\eqref{4.69} we have
\begin{equation}
(\widetilde{\bf s}g)(z)=\sum_{j=1}^d 
\widetilde{B}^*(I-Z_\cX(z)\widetilde{A}^*)^{-1}
{\mathcal I}_{j} x_k=\widetilde{B}^*(I-Z_\cX(z)\widetilde{A}^*)^{-1}
{\bf x}.
\label{4.80}
\end{equation}
Upon evaluating the latter equality at $z=0$ we get for the operator
$B^*$ 
from \eqref{4.77}
$$
B^*g=\tB^*(\oplus_{j=1}^d U^*)g=\tB^*{\bf x}=(\widetilde{\bf s}g)(0)
$$
so that the formula \eqref{4.22} for $B^*$ holds. We also have from
\eqref{4.80}
$$
(\widetilde{\bf s}g)(z)-(\widetilde{\bf s}g)(0)=
\widetilde{B}^*(I-Z_{\cX}(z)^*\widetilde{A}^*)^{-1}Z_{\cX}
(z)^*\widetilde{A}^*{\bf x}.
$$
On the other hand, for the operator $A$ defined in \eqref{4.77}, we
have
\begin{align*}
{\bQ}(z)^*\otimes A^*g&={\bQ}(z)^*\otimes (A^*(\oplus_{j=1}^d U){\bf 
x})\\
&={\bQ}(z)^*\otimes U (\tA^*{\bf x})=(\oplus_{j=1}^d 
U)Z_\cX(z)^*\tA^*{\bf x}
\end{align*}
and therefore, by formula \eqref{4.80} applied to 
$Z_\cX(z)^*\tA^*{\bf x}$ instead of  ${\bf x}$ we get
$$
\widetilde{\bf s}(Z_{\cH({\mathbb 
K})}(z)^*A^*g)(z)=\widetilde{B}^*(I-Z_\cX(z)^*\widetilde{A}^*)^{-1}Z_\cX
(z)\widetilde{A}^*{\bf x}
$$
which together with \eqref{4.79} implies \eqref{4.21}. This completes 
the proof of Step 3. 

\smallskip

To complete the proof of the theorem, it suffices to observe that
the colligation $\widetilde{\bU}$ is unitarily equivalent to a 
$\tcfm$ colligation $\bU$ by construction \eqref{4.77} and definition 
\eqref{1.17}. 
\end{proof}

\section{Characteristic functions of commutative row contractions: 
unitary equivalence and coincidence}
\label{S:CharF}

The  Sz.-Nagy-Foias characteristic function of a Hilbert space 
contraction  $T\in\cL(\cX)$ is defined as 
$$
\theta_{T}(z)= \left.\left( -T + z D_{T^{*}}(I_{\cX} -
 z T^{*})^{-1} D_{T}\right)\right|_{{\mathcal D}_{T}}
$$
where $D_{T}$ and $D_{T^{*}}$ are the defect operators and 
${\mathcal D}_{T}$, ${\mathcal D}_{T^*}$ are the defect spaces 
recalled in \eqref{6.1} below. The function $\theta_{T}$
belongs 
to the Schur class ${\mathcal S}({\mathcal D}_{T}, {\mathcal 
D}_{T^{*}})$ and is {\em pure} in the sense that
\begin{equation}
\|\theta_{T}(0)u\| = \| u \| \quad \text{for some} \quad u \in \cU\quad
\Longrightarrow \quad u = 0.
\label{6.01}
\end{equation}
The Schur-class membership and pureness are the properties which 
characterize characteristic functions. 
A classical result of 
B.~Sz.-Nagy and C.~Foias (see \cite{NF}) is: {\em  If $T$ is
completely 
nonunitary (\textbf{c.n.u.}) contraction
(i.e., $T$ is a contraction and there is no nontrivial  reducing 
subspace $\cM$ for $T$ so that the $T\vert_{\cM}$ is unitary), then 
the characteristic function $\theta_T$ is a complete unitary
invariant of 
$T$}. More precisely, if $T\in\cL(\cX)$ and $R\in\cL(\widetilde{\cX})$
are two \textbf{c.n.u.}~contractions, then they are unitarily 
equivalent if and only if their characteristic functions $\theta_{T}$ 
and $\theta_{R}$ coincide;
by definition two operator-valued functions $S$ and $\widetilde{S}$
with the same domain of definition {\em coincide} if 
$S(z)\equiv \alpha\widetilde{S}(z)\beta$ for some unitary 
transformations $\alpha$ and $\beta$. Moreover, if one starts with a 
pure Schur-class function $S$, one can associate a \textbf{c.n.u.}~ 
contraction $T(S)$ defined on the associated Sz.-Nagy-Foias
canonical model Hilbert space ${\mathcal K}(S)$. We mention that
there is a 
dictionary between the Sz.-Nagy--Foias model $(T(S), {\mathcal K}(S))$
and the de Branges-Rovnyak model space $(T_{dBR}(S), \cH(\widehat
K_{S}))$
where $\widehat K_{S}$ is the positive kernel given by \eqref{1.2}
and 
where $T_{dBR}(S) = A^{*}$ where $A$ is the operator on $\cH(\widehat 
K_{S})$ given in part (3) of Theorem  \ref{T:sum}
(see e.g.~\cite{BK}).  It is easy to see that if 
$T$ and $R$ are unitarily equivalent, then the associated model 
contraction operators $T(\theta_{T})$ and $R(\theta_{R})$ (or 
$T_{dBR}(\theta_{T})$ and $T_{dBR}(\theta_{R})$) are unitarily
equivalent.
The result mentioned above can be further elaborated as follows:  {\em If $T\in\cL(\cX)$ 
is a \textbf{c.n.u.}~contraction operator, then $T$ is unitarily 
equivalent to its functional model contraction operator 
$T(\theta_{T})$ on $\cK(\theta_{T})$ or $T_{dBR}(\theta_{T})$ on 
$\cH(\widehat{K}_{\theta_{T}})$.}  We should also mention that if $T$ is {\em 
completely non-coisometric} (\textbf{c.n.c.}---see the discussion 
below for precise definitions), then it suffices to take $T_{dBR}(S) 
= A^{*}$ where $A$ is the backward shift operator acting on $\cH(K_{S})$ 
as in \eqref{1.8}.  We mention the paper of Foias-Sarkar \cite{FS} as 
a very recent application of the Sz.-Nagy-Foias model theory for a 
single contraction operator.
Let us also mention that there is a second approach to unitary 
classification of Hilbert space operators based on the curvature 
invariant of Cowen and Douglas \cite{CD}; for a recent comparison 
between these two approaches, we refer to \cite{DKKS}.

In this section we discuss extensions of 
this Sz.-Nagy--Foias model theory to the context of 
row contractions, that is to $d$-tuples of operators ${\mathbf T} = 
(T_{1}, \dots, T_{d})$ on a Hilbert space $\cX$ for which the 
associated block-row matrix is contractive:
\begin{equation}
 \|T\| \le 1 \quad\text{where}\quad T  = \begin{bmatrix} T_{1} &
\cdots &T_{d} \end{bmatrix} \colon \cX^{d} \to \cX.
\label{6.0}
\end{equation}
For such a row-contraction, let
\begin{align}
D_{T} &= (I_{\cX^{d}} - T^{*}T)^{1/2}, \qquad {\mathcal D}_{T} =
\overline{\operatorname{Ran}}\, D_{T} \subset \cX^{d},  \notag\\
D_{T^{*}} &= (I_{\cX} - T T^{*})^{1/2}, \qquad
{\mathcal D}_{T^{*}} = \overline{\operatorname{Ran}}\, 
D_{T^{*}}\subset \cX.
\label{6.1}
\end{align}
The  characteristic function $\theta_{T, nc}$ of a row contraction 
has been introduced in \cite{popescujot} (in slightly different terms) as 
a formal power series in $d$ noncommuting indeterminates
$\bz_1,\ldots, \bz_d$ which can be written in a compact realization form as
\begin{equation}  \label{6.2}   
\theta_{T, nc}(\bz)  = \left.\left( -T + D_{T^{*}}(I_{\cX} -
Z_\cX(\bz)T^{*})^{-1}Z_\cX(\bz) D_{T}\right)\right|_{{\mathcal D}_{T}}
\colon
{\mathcal D}_{T} \to{\mathcal D}_{T^ {*}}
\end{equation}
where $Z_\cX(\bz)$ is of the form \eqref{1.12a} (but with the 
noncommuting indeterminates $\bz_{1}, \dots, \bz_{d}$ replacing the 
commuting variables $z_{1}, \dots, z_{d}$).
To write this expression out more explicitly, we need the following 
notation connected with formal power series in noncommuting 
indeterminates.  Let ${\mathcal F}_{d}$ consists of all words $v = 
i_{N} \cdots i_{1}$ with letters $i_{j}$ coming from the alphabet 
$\{1, \dots, d\}$. Then the operator of concatenation
$ v \cdot v' = v''$ where
$$
 v'' = j_{N'} \cdots j_{1} i_{N} \cdots i_{1} \quad\text{if}\quad
 v' = j_{N'} \cdots j_{1} \quad\text{and}\quad v = i_{N} \cdots i_{1}
$$
makes ${\mathcal F}_{d}$ a free semigroup; here we include the empty 
word, denoted as $\emptyset$, as an element of ${\mathcal F}_{d}$ 
which serves as the identity element of the semigroup.  For $\{\bz_{1}, 
\dots, \bz_{d}\}$ a $d$-tuple of freely noncommuting indeterminates and 
for $v = i_{N} \cdots i_{1}$ an element of ${\mathcal F}_{d}$, we let
$\bz^{v}$ denote the noncommutative monomial $\bz_{i_{N}} \cdots
\bz_{i_{1}}$. For the case $v = \emptyset$, we set $\bz^{\emptyset} = 1$.
We extend this noncommutative functional calculus to a $d$-tuple of
operators ${\mathbf T} = (T_{1}, \dots, T_{d})$ on a Hilbert space
$\cX$:
\begin{equation}  \label{bAv}
{\mathbf T}^{v} = T_{i_{N}} \cdots T_{i_{1}}\quad \text{if}\quad
v = i_{N} \cdots i_{1} \in \cF_{d} \setminus \{ \emptyset\};\quad {\mathbf
T}^{\emptyset} = I_{\cX}.
\end{equation}
Similarly, for ${\mathbf T}^{*} = (T_{1}^{*}, \dots, T_{d}^{*})$ equal to a 
$d$-tuple of (not necessarily commuting) operators, we use the 
notation ${\mathbf T}^{*v}$  to indicate the product ${\mathbf T}^{*v} = T_{i_{N}}^{*} 
\cdots T_{i_{1}}^{*}$, with ${\mathbf T}^{* \emptyset}$ equal to the identity 
operator $I$.
We wish to point out that the expression \eqref{6.2} for the 
characteristic function can also be written more explicitly as
the noncommutative formal power series 
\begin{equation}  \label{NCcharfunc}
\theta_{T, nc}(\bz) =  \sum_{v \in {\mathcal F}_{d}} [ \theta_{T,nc}]_{v}\bz^{v}
\end{equation}
where the power series coefficients $[ \theta_{T,nc}]_{v}: \; \cD_{T}
    \to \cD_{T^{*}}$ are given by
\begin{equation}
[\theta_{T, nc}]_{v} =  \begin{cases}
-T|_{{\mathcal D}_{T}} & \text{if } v = \emptyset, \\
D_{T^{*}} {\mathbf T}^{*v'} {\mathcal I}_{j}^{*} D_{T}
& \text{if } v \ne \emptyset \; \text{ has the form } \; v=  v' \cdot j
    \end{cases}
    \label{moments}
\end{equation}
(where ${\mathcal I}_{j}^{*} \colon \cH^{d} \to \cH$ is as in
\eqref{1.13}). Thus we see that knowledge of the characteristic function amounts to 
knowledge of all the moment operators \eqref{moments}. 
It is readily seen that 
if $T = \begin{bmatrix}T_{1} & \dots & T_{d}\end{bmatrix}$ and $R = \begin{bmatrix} 
R_{1} & \dots & R_{d}\end{bmatrix}$ are two unitarily equivalent row
contractions (i.e., $UT_iU^*=R_i$ for $i=1,\ldots,d$ and some unitary
operator $U$), then the formal power series $\theta_{T,nc}(\bz)$ and 
$\theta_{R,nc}(\bz)$ coincide (the coincidence for noncommutative formal
power series is defined in much the same way as for usual functions), 
or equivalently, the set of moments \eqref{moments} associated with 
$T$ coincide with those associated with $R$. The converse was proved in 
\cite{popescujot} under the assumption that $T$ and $R$ are {\em
completely non-coisometric} ({\bf c.n.c.}). Recall that  a row contraction 
$T$ as in \eqref{6.0} is called {\em completely  
non-coisometric} ({\bf c.n.c.}) if  there is no nontrivial  subspace ${\mathcal 
M}\subset\cX$ invariant under 
$T_i^*$ for $i=1,\ldots,d$ so that that the operator
\begin{equation}
P_{\mathcal M}\begin{bmatrix} T_1\vert_{\mathcal 
M}&\ldots&T_d\vert_{\mathcal
M}\end{bmatrix}: \; {\mathcal M}^d\to {\mathcal M}
\label{6.4}
\end{equation}
is a coisometry.  An equivalent formulation is that
\begin{equation}   \label{cnc-span}
\cX = \bigvee \left\{ \operatorname{Ran} {\mathbf T}^{v}  D_{T^{*}} \colon v 
\in {\mathcal F}_{d}, k = 1, \dots, d \right\}.
\end{equation}
Thus, the result from \cite{popescujot} states that
if $T$ is a {\bf c.n.c.}~row contraction, then  the
characteristic function $\theta_{T, nc}$  is a complete unitary invariant for 
$T$. In view of the explicit formula  \eqref{NCcharfunc}--\eqref{moments} for 
$\theta_{T,nc}$, we see that the latter result can be rephrased as saying that the set of 
moments 
\begin{equation}  \label{moments'}
 -T^{*}, \, D_{T^{*}} {\mathbf T}^{v} {\mathcal I}_{j}^{*} D_{T} \colon \cD_{T} \to 
\cD_{T^{*}}
\end{equation}
(where $v$ runs over all words in ${\mathcal F}_{d}$ and $j$ runs over 
all indices in $\{1, \dots, d\}$) form a complete set of unitary invariants 
for a row contraction in the {\bf c.n.c.}~case.

\smallskip

The more general class of completely nonunitary ({\bf c.n.u.})~row 
contractions as defined in \cite{Cuntz-scat} consists of those $T$ for which there is no 
nontrivial subspace ${\mathcal M}$ reducing for each $T_1,\ldots,T_d$ on which the
 operator block-row matrix 
\eqref{6.4} is unitary.  An equivalent formulation (see 
\cite[Proposition 3.3.5]{Cuntz-scat}) is that 
\begin{equation}   \label{cnu-span}
\cX = \bigvee\left\{ \operatorname{Ran} {\mathbf T}^{v}  D_{T^{*}}, 
\operatorname{Ran} {\mathbf T}^{\alpha} {\mathbf T}^{*\beta} 
{\mathcal I}_{k}^{*} D_{T}\colon \alpha, 
\beta \in {\mathcal F}_{d},\, k=1, \dots, d  \right\}.  
\end{equation}
The characteristic function $\theta_{T,nc}$ does 
not recover $T$ up to unitary equivalence. However, it was shown in 
\cite{Cuntz-scat} that there is an operator $L_T$ so that
the pair $(\theta_{T,nc}, L_T)$ is a  complete unitary invariant
for $T$. A more concrete version of the result from
\cite{Cuntz-scat} (see the 
discussion around equations (5.3.6) there) is that a 
complete set of invariants (up to coincidence) for the 
\textbf{c.n.u.}~case is given by the expanded set of moments 
\begin{equation}   \label{exp-moments}
    -T, \, D_{T^{*}} {\mathbf T}^{*v} {\mathcal I}_{j}^{*} D_{T} \colon
    \cD_{T} \to \cD_{T^{*}}, \quad
    D_{T} {\mathcal I}_{k} {\mathbf T}^{v} {\mathbf T}^{*v'} 
    {\mathcal I}_{j}^{*} D_{T} \colon \cD_{T} \to \cD_{T}
\end{equation}
where $v$ and $v'$ run over all words in ${\mathcal F}_{d}$ and where
$k$
and $j$ run over the set of indices $\{1, \dots, d\}$.  
This work also makes explicit the 
construction of a model contraction operator acting on a 
Sz.-Nagy--Foias canonical model space
(see \cite{popescujot} for the \textbf{c.n.c.}~case and 
\cite{Cuntz-scat} for the \textbf{c.n.u.}~case). 

\smallskip

It is not difficult to see that any such characteristic function 
$\theta_{T, nc}$ defines a contractive multiplier on the Fock space which 
commutes with the right creation operators (see \cite{Cuntz-scat}).  
Formal power series for which the associated multiplication operator 
is bounded on the Fock space are called {\em multianalytic functions} 
in \cite{popescujot}.  Conversely, any contractive multianalytic 
function $S(\bz) = \sum_{v \in {\mathcal F}_{d}} S_{v} \bz^{v}$ is a 
characteristic function for some \textbf{c.n.u.} row contraction $T$ 
if and only if $S$ is also {\em pure} in the sense that 
$\| S_{\emptyset} u\|_{\cY} = \| u \|_{\cU}$ only when $u = 0$ (see 
\cite[page 89]{Cuntz-scat}).  
Contractive multipliers $S(\bz) = \sum_{v \in {\mathcal F}_{d}} S_{v} \bz^{v}$
equal to the characteristic function of a \textbf{c.n.c.}~row 
contraction $T$ are characterized by having the additional property that
$I - S(\bz)^{*} S(\bz) \ge G(\bz)^{*} G\bz)$ for some multianalytic $G$ 
forces $G = 0$ (see Remark 5.3.5 in \cite{Cuntz-scat}).  
Thus one can say that \textbf{c.n.c.}~row 
contractions are parametrized by (equivalence classes up to 
coincidence of) pure contractive multianalytic functions $S(\bz)$ for 
which the defect $I - S(\bz) S(\bz)^{*}$ has zero maximal factorable 
minorant, while \textbf{c.n.u.}~row contractions are parametrized by 
equivalence classes of pure contractive multianalytic functions 
combined with the second invariant $L_T$, the details of which need 
not concern us here. 

\smallskip

If we replace the noncommutative indeterminates $\bz_1,\ldots, \bz_d$ 
with commuting variables $z_{1}, \dots, z_{d}$
in  formula \eqref{6.2}, then we get a function $\theta_{T}(z)$ analytic on $\B^d$ 
and certainly depending 
on $T$ only. Moreover, this function is the characteristic function
of the colligation
\begin{equation}
\bU_{T}=\begin{bmatrix}T^* & D_{T} \\ D_{T^{*}} &
-T\end{bmatrix}: \;
\begin{bmatrix}\cX \\ {\mathcal D}_{T}\end{bmatrix}\to
\begin{bmatrix}\cX^d \\ {\mathcal D}_{T^*}\end{bmatrix}
\label{6.3}
\end{equation}
which is the {\em Halmos unitary dilation} of $T^*$; 
therefore, $\theta_{T}$ belongs to ${\mathcal S}_{d}({\mathcal 
D}_{T}, {\mathcal D}_{T^{*}})$ by Theorem \ref{T:1.1}. 
It is not hard to 
show that an $\cL(\cU,\cY)$-valued function $S$ coincides with a
function  $\theta_T$ of the form \eqref{6.2} for some Hilbert-space 
row contraction  $T$ of the form \eqref{6.0}
if and only if $S$ belongs to the Schur class $\cS_d(\cU,\cY)$
and is  pure in the sense of \eqref{6.01}. Thus, the commutative
(analytic) version of formula \eqref{6.2} perfectly fits the
framework of the present paper. However, this version is meaningful only in case 
$T=\begin{bmatrix} T_1 & \dots & T_d\end{bmatrix}$  is a {\em commutative row-contraction} 
(i.e., $T_iT_j=T_jT_i$ for 
$i,j=1,\ldots,d$); otherwise  $\theta_T$ does not contain enough 
information about $T$ to recover $T$ up to unitary equivalence.  
In what follows, we focus on commutative row-contraction $T$ of the form \eqref{6.0};  
following \cite{BES}, \cite{BES2}, 
we refer to the (analytic) 
function \eqref{6.2} (with commuting variables $z_{1}, \dots, z_{d}$ 
in place of the noncommuting indeterminates $\bz_{1}, \dots, \bz_{d}$) as the 
{\em the characteristic function} $\theta_{T}(z)$ of such a
$T$ and the formal power series in freely noncommuting indeterminates 
$\theta_{T, nc}(\bz)$ as in \eqref{6.2} as the {\em noncommutative} (or {\em n.c.}) {\em 
characteristic function} of $T$.  

\smallskip

For the case of {\em commutative} row- contractions, note that the conditions 
\eqref{cnc-span} and \eqref{cnu-span} simplify.  Thus {\em the commutative row contraction 
$T$ is} \textbf{c.n.c.}~{\em exactly when}
\begin{equation}   \label{cnc-span-com}
    \cX = \bigvee \left\{ \operatorname{Ran} {\mathbf T}^{n}  D_{T^{*}} \colon n
\in {\mathbb Z}^{d}_{+}, k = 1, \dots, d \right\},
\end{equation}
where we use the standard multivariable notation 
\begin{equation}   \label{power}
{\mathbf T}^{n} = T_{1}^{n_{1}} \cdots T_{d}^{n_{d}}\quad\mbox{for}\quad 
n = (n_{1}, \dots, n_{d})\in {\mathbb Z}^{d}_{+}
\end{equation}
and for ${\mathbf T} = (T_{1}, \dots, T_{d})$ a 
commutative operator tuple.  Similarly, {\em the commutative row-contraction $T$
of the form \eqref{6.0} is } \textbf{c.n.u.}~{\em exactly when}
\begin{equation}   \label{cnu-span-com}
    \cX = \bigvee\left\{ \operatorname{Ran} {\mathbf T}^{n}  D_{T^{*}}, 
\operatorname{Ran} {\mathbf T}^{n} {\mathbf T}^{*m} {\mathcal I}_{k} D_{T}\colon n, 
m \in {\mathbb Z}^{d}_{+},\, k=1, \dots, d  \right\}.  
\end{equation}

\begin{remark}\label{R:com-vs-nc}
    {\em
Before proceeding further, let us first observe that {\em any 
commutative row-unitary}  (or even just {\em row-isometric}) {\em tuple ${\mathbf U}$ as in \eqref{6.0} is 
trivial if $d>1$}.  Indeed, suppose that  $d>1$ and 
${\mathbf U} = (U_{1}, \dots, U_{d})$ is  a row-isometric 
tuple.  This means that each $U_{j}$ is an isometry and the ranges of 
$U_{1}, U_{2}, \dots, U_{d}$ have pairwise orthogonal ranges 
(spanning the whole space ${\mathcal X}$ in case ${\mathbf U}$ is row-unitary).
In particular,
$\operatorname{Ran} U_{j} \perp \operatorname{Ran} U_{k}$ for $j \ne 
k$.  But then we also have
$$\operatorname{Ran} U_{j}U_{k} = \operatorname{Ran} U_{k} U_{j} 
\subset 
\operatorname{Ran} U_{j} \cap \operatorname{Ran} U_{k} = \{0\} \text{ 
for } j \ne k.
$$
As $U_{k}U_{j}$ is also an isometry, it follows that the ambient 
Hilbert space ${\mathcal X}$ is the zero space.
As a consequence of this observation, it follows that {\em any 
commutative row contraction $T$ is} \textbf{c.n.u.}.
Indeed, there can be no nonzero reducing subspace for $\bT$ 
on which $\bT$ is row-unitary, since then necessarily the 
restriction of ${\mathbf T}$ to such a subspace would have to be simultaneously 
commutative and non-commutative.  We conclude that {\em any 
commutative row-contraction $T$ as in \eqref{6.0} 
is unitarily equivalent to a noncommutative Sz.-Nagy-Foias functional 
model as in \cite{Cuntz-scat} based on its n.c.-characteristic 
function $\theta_{\bT,nc}$.}  The drawback of this model of course is that 
it does not display prominently the additional structure that $T$ is commutative.}
\end{remark}

The following result was first obtained in \cite{BES2};  it is also 
possible to give a unified proof which includes the noncommutative 
and commutative setting in one formalism (see \cite{BhatBhat,
Popescu2006, Popescu2007}) and there is now an extension of the 
general theory to the setting of more general operator-tuples associated with a general 
``positive regular freely holomorphic function'' $f$ (see 
\cite{Popescu2010}. 
We  give here an alternative direct proof of the result
based on the results from Section 2 which suggests extensions to the cases 
beyond the \textbf{c.n.c.}~setting.

\begin{theorem} 
\label{P:6.4}
Two commutative {\bf c.n.c.}~row contractions 
$T=\begin{bmatrix}T_1&\dots&T_d\end{bmatrix}$ and 
$R=\begin{bmatrix}R_1&\dots&R_d\end{bmatrix}$ 
are unitarily equivalent if and only if their  characteristic 
functions $\theta_T$ and $\theta_{R}$ coincide.  
\end{theorem}

\begin{proof}
We prove the nontrivial ``if'' part. We first observe that a
commutative 
row contraction $T$ is completely  non-coisometric
if and only if the colligation \eqref{6.3} is observable, i.e.,
$$
D_{T^*}(I_\cX-Z_{\cX}(z)T^*)^{-1}x\equiv 0\; \; \Longrightarrow \; \;
x=0.
$$
Indeed, the latter implication can be equivalently written as 
\begin{equation}
D_{T^*}{\bf T}^{*n}x=0 \quad\mbox{for all}\quad n\in{\mathbb Z}_+^d\;
\;
\Longrightarrow \; \; x=0,
\label{6.6}
\end{equation}
as can be seen from the expansion
$$
D_{T^*}(I_\cX-Z_{\cX}(z)T^*)^{-1}=D_{T^*}\left(I_\cX-\sum_{j=1}^d
z_jT^*_j\right)^{-1}
=\sum_{n\in{\mathbb Z}_+^d}\frac{|n|!}{n!}D_{T^*}{\bf T}^{*n}z^n,
$$
where we have used notation \eqref{mnot} and \eqref{power}. On the
other hand,
$$
{\mathcal M}:=\left\{x\in\cX: \; D_{T^*}{\bf T}^{*n}x=0 \quad\mbox{for
all}\quad n\in{\mathbb Z}_+^d\right\}
$$
is the maximal ${\bf T}^*$-invariant subspace of $\cX$ such that the
operator \eqref{6.4} is a coisometry. Combining this with
observability
characterization \eqref{6.6} and the definition of a ${\bf c.n.c.}$
tuple, we get the desired equivalence.

\smallskip

Let us  assume that $\theta_T$ and $\theta_R$ coincide, i.e., that 
\begin{equation}
\theta_T(z)=\alpha \theta_R(z)\beta^*
\label{6.13}   
\end{equation} 
where $\alpha \colon {\mathcal D}_{R^{*}} \to {\mathcal D}_{T^{*}}$
and
$\beta \colon{\mathcal D}_{R} \to {\mathcal D}_{T}$ are unitary
operators.
Thus,
\begin{align*}
\theta_T(z)&=\left.\left( -T + D_{T^*}(I_{\cX} -
Z_{\cX}(z)T^{*})^{-1}Z_{\cX}(z)
D_{T}\right)\right|_{{\mathcal D}_{T}}\\
&=\left.\left( -\alpha R\beta^* + \alpha D_{R^{*}}(I_{\widetilde\cX} -
Z_{\widetilde\cX}(z)T^{*})^{-1}Z_{\widetilde\cX}(z) 
D_{R}\beta^*\right)\right|_{{\mathcal D}_{R}}
\end{align*}
and we have two commutative unitary colligations
\begin{equation}
\bU_1=\begin{bmatrix}T^{*} & D_{T} \\  D_{T^*} &
-T\end{bmatrix}\quad\mbox{and}\quad
\bU_2=\begin{bmatrix}R^{*} & D_{R}\beta^* \\  \alpha D_{R^*} &
-\alpha R\beta^*\end{bmatrix}
\label{6.7}
\end{equation}
with the same input and output spaces and with the same
characteristic  function $\theta_{T}$.  

Since $T$ and $R$ are completely non-coisometric, these colligations are both
observable. As the lower diagonal entries in $\bU_1$ and $\bU_2$ are equal (evaluate
\eqref{6.13} at $z=0$), Corollary 3.7 from \cite{BBF2b} implies that $T^*$ is unitarily
equivalent to $R^*$.
\end{proof} 

\smallskip

We now discuss how our \textbf{t.c.f.m.}~colligations can be used to 
study unitary equivalence and unitary invariants for row contractions 
more general that \textbf{c.n.c.} Before proceeding further, 
let us recall that any 
unitary realization $\widetilde{\bU}=\sbm{\tA & \tB \\ \tC & D}$
\eqref{4.68} of an $S\in\cS_d(\cU,\cY)$ induces an Agler
decomposition \eqref{4.3}
for $S$ via formulas \eqref{4.74}, \eqref{4.75}. If $T$ is a row 
contraction, then the unitary realization $\bU_T$ \eqref{6.3} for 
$\theta_T$ induces the Agler decomposition  
\begin{equation}
{\mathbb K}_T(z,\zeta)={\mathbb G}_T(z){\mathbb 
G}_T(\zeta)^*,\quad\mbox{where}\quad 
{\mathbb G}_T(z)=\left[\begin{array}{c}
D_{T^*}(I_\cX-Z_{\cX}(z)T^*)^{-1} \\
D_{T}(I_{\cX^{d}}-Z_{\cX}(z)^*T)^{-1}{\mathcal I}_{1} \\ \vdots \\
D_{T}(I_{\cX^{d}}-Z_{\cX}(z)^*T)^{-1}{\mathcal I}_{d} \end{array}\right].
\label{6.11}
\end{equation}
Moreover, in case the Halmos-dilation colligation \eqref{6.3} is 
closely connected, Theorem \ref{T:4.7} tells us that $\bU_{T}$ is 
unitarily equivalent to some \textbf{t.c.f.m.}~colligation associated 
with the Agler decomposition ${\mathbb K}_{T}$ \eqref{6.11} for 
$\theta_{T}$.  Let us write $\cD(T)$ and $\cR(T)$ for the domain and 
range of the isometry $V$ given by \eqref{4.11} for the case where 
${\mathbb K} = {\mathbb K}_{T}$; we also write $V_{T}$ rather than 
$V$ for this case.
Then any \textbf{t.c.f.m.}~colligation associated with ${\mathbb 
K}_{T}$ has the 
form
\begin{equation}   \label{bU1}
    \bU = \begin{bmatrix} A & B  \\ C & D \end{bmatrix} \colon 
    \begin{bmatrix} \cH({\mathbb K}) \\ \cD_{T} \end{bmatrix} \to
 \begin{bmatrix} \cH({\mathbb K})^{d} \\ \cD_{T^{*}} \end{bmatrix}
\end{equation}
where, upon the identifications 
$$
\begin{bmatrix} \cH({\mathbb 
K}) \\ \cD_{T} \end{bmatrix}  \sim  
\begin{bmatrix} \cR(T)^{\perp} \\ \cR(T) \oplus \cD_{T} 
\end{bmatrix}  \text{ and } 
 \begin{bmatrix} \cH({\mathbb K})^{d} \\ \cD_{T^{*}} \end{bmatrix}
\sim  
 \begin{bmatrix} \cD(T)^{\perp} \\ \cD(T) \oplus \cD_{T^{*}} 
\end{bmatrix} 
$$ 
as in Lemma \ref{L:2.5}, $\bU^{*}$ has the form
\begin{equation}   \label{bU2}
    \bU^{*} = \begin{bmatrix} X_{T} & 0 \\ 0 & V_{T} \end{bmatrix} 
    \colon
 \begin{bmatrix} \cD(T)^{\perp} \\ \cD(T) \oplus \cD_{T^{*}}
\end{bmatrix}
  \to \begin{bmatrix} \cR(T)^{\perp} \\ \cR(T) \oplus \cD_{T}
\end{bmatrix}.
\end{equation}
The fact that $\bU_{T}$ is unitary implies that $\bU$ is unitary and
hence also $\dim \cD(T)^{\perp} = \dim \cR(T)^{\perp}$ and $X_{T} \colon
\cD^{\perp}
\to \cR^{\perp}$ is unitary.  We conclude:

\begin{proposition}
\label{P:prel}
Given a row-contraction $T$ such that the the Halmos-dilation colligation \eqref{6.3} is
closely connected and given an Agler decomposition
${\mathbb K}_{T}$ for $\theta_{T}$, there is a choice of unitary
$X_{T}$ from $\cD(T)^{\perp}$ to $\cR(T)^{\perp}$ so that $T$ is unitarily
equivalent to $A(T)^{*}$, where $A(T)$ is
determined just from $(\theta_{T}, {\mathbb K}_{T}, X_{T})$ via the
decomposition \eqref{bU1} for $\bU$ defined by \eqref{bU2}.
\end{proposition}

This suggests that we define a new invariant consisting of a triple 
of objects $(S, {\mathbb K}, X)$ defined 
as follows.  We let $S$ be any pure Schur-class function in 
$\cS_{d}(\cU, \cY)$.  We then let ${\mathbb K}$ be any Agler 
decomposition for $S$.  The remaining ingredient to form a 
\textbf{t.c.f.m.}~colligation associated with the Agler decomposition 
${\mathbb K}$ for $S$ is a choice of contraction operator $X \colon 
\cD^{\perp} \to \cR^{\perp}$. Part of our admissibility requirements on $(S, {\mathbb 
K}, X)$ is that 

\smallskip

(1) {\em  it turns out that $\dim \cD^{\perp} = \dim 
\cR^{\perp}$}. 

\smallskip

In this case there exists a unitary operator $X \colon 
\cD^{\perp} \to \cR^{\perp}$ which then defines completely a 
\textbf{t.c.f.m.}~ associated with ${\mathbb K}$ and $S$ which gives 
a {\em unitary} realization of $S$.  As the 
final admissibility requirement, we demand that 

\smallskip

(2) {\em the choice of unitary 
$X \colon \cD^{\perp} \to \cR^{\perp}$ is such that the $d$-tuple of 
operators $(A_{1}, \dots, A_{d})$ constructed from the decomposition 
\eqref{bU1} for $\bU$ is commutative}.  

\smallskip

Let us call any such triple 
$(S, {\mathbb K}, X)$ (consisting of a pure Schur-class function $S$,
an 
Agler decomposition ${\mathbb K}$ for $S$, and a unitary operator $X 
\colon \cD^{\perp}$ to $\cR^{\perp}$ such that the admissibility 
requirements (1) and (2) are also satisfied) an {\em admissible
triple}. In the discussion above we explained how to attach a particular 
admissible triple $(\theta_{T}, {\mathbb K}_{T}, X_{T})$ (the {\em 
characteristic admissible triple} of $T$) to any commutative row
contraction  $T$ for which $\bU_{T}$ is closely connected.  

It is not difficult to characterize  when the colligation $\bU_{T}$ 
is closely connected directly in terms of $T$.  Toward this end, 
given a  commutative row contraction $T$, introduce the subspace
\begin{align}   
    {\mathcal M}^{(1)}_T = &  \left\{x\in\cX \colon   
    D_{T^*}(I_{\cX}-Z_{\cX}(z)T^*)^{-1}x\equiv 0 
\text{ and } \right. \notag \\ 
  & \quad D_{T}(I_{\cX^{d}}-Z_{\cX}(z)^*T)^{-1}{\mathcal I}_{i}x\equiv 0 \; 
 \text{ for } \; i=1, \dots, d \}.
 \label{5.15}
\end{align}
The space ${\mathcal M}^{(1)}_T$ is the orthogonal 
complement in $\cX$ of the space $ \cH^{\mathcal 
O}_{D_{T^*},T^*}\bigvee\cH^{\mathcal C}_{T^*,D_{T}}$ (see definitions 
\eqref{1.15}); thus ${\mathcal M}^{(1)}_T=\{0\}$ if and only if 
the colligation \eqref{6.3} is closely connected. For the case $d=1$, 
the condition ${\mathcal M}^{(1)}_T=\{0\}$ simply means that $T$ is 
completely nonunitary.  With this as motivation, we make the 
following definition.

\begin{definition}  \label{D:6.5'} We say that the commutative row-contraction 
$T=\begin{bmatrix} T_1 & \dots & T_d\end{bmatrix}$ is {\em 
closely connected} (\textbf{c.c.}) 
   if ${\mathcal M}^{(1)}_T=\{0\}$ where $\cM_{T}^{(1)}$ is as in 
   \eqref{5.15}.
Equivalently, $T$ is  \textbf{c.c.}~if and only if
\begin{equation}  \label{span-cc}
    \cX = \left\{ \operatorname{Ran} {\mathbf T}^{n}D_{T^{*}}, \, 
    \operatorname{Ran} {\mathcal I}_{k}^{*}X_{n}D_{T}\colon n \in 
    {\mathbb Z}^{d}_{+}, k=1, \dots, d \right\}
 \end{equation}
 where $X_{n}$ is given by 
 \begin{equation}   \label{Xn}
     (I_{\cX^{d}} - T^{*} Z_{\cX}(z))^{-1} = 
    \sum_{n \in {\mathbb Z}^{d}_{+}} X_{n} z^{n}.
 \end{equation}
   \end{definition}
We note that the $X_{n}$'s in \eqref{Xn} are difficult to compute 
explicitly in general.  Nevertheless it is clear that
\begin{align}
&  \bigvee \left\{ \operatorname{Ran}{\mathbf T}^{n} D_{T^{*}} \colon n 
\in {\mathbb Z}^{d}_{+} \right\} \notag \\
& \quad \subset \bigvee \left\{ \operatorname{Ran}{\mathbf T}^{n} D_{T^{*}}, 
\operatorname{Ran}{\mathcal I}_{k}^{*} X_{n}D_{T} \colon n \in {\mathbb 
Z}^{d}_{+}, k=1, \dots, d \right\}  
\notag \\
& \quad \subset \bigvee \left\{ \operatorname{Ran}{\mathbf T}^{*n} D_{T^{*}}, 
\operatorname{Ran} {\mathbf T}^{n} {\mathbf T}^{*m} {\mathcal I}_{k} D_{T}\colon n, 
m \in {\mathbb Z}^{d}_{+},\, k=1, \dots, d  \right\} = \cX,
\label{chain1}
\end{align}
from which it follows that 
\textbf{c.n.c.}~$\Rightarrow$ \textbf{c.c.}~$\Rightarrow$ 
\textbf{c.n.u.}~(since \textbf{c.n.u.}~holds for {\em any} 
commutative row contraction by Remark \ref{R:com-vs-nc}).
Henceforth, unless otherwise stipulated, we assume that 
{\em $T$ is a \textbf{c.c.}~commutative row contraction}.

\smallskip

We next observe that if two \textbf{c.c.}~commutative row
contractions $T$  and $R$ are unitarily 
equivalent, then  the associated Agler decompositions ${\mathbb K}_T$ 
and ${\mathbb K}_R$ together with the
characteristic functions $\theta_{T}$ and $\theta_{R}$ defined as 
in \eqref{6.11} jointly coincide in the sense that 
\begin{equation}
\theta_{T}(z) = \alpha \theta_{R}(z) \beta^{*}, \quad
{\mathbb K}_T(z,\zeta)= \begin{bmatrix}\alpha & 0 \\ 0 &
\bigoplus_1^d 
\beta\end{bmatrix}{\mathbb K}_R(z,\zeta)\begin{bmatrix}\alpha^* & 0
\\ 0 & 
\bigoplus_1^d \beta^*\end{bmatrix}
\label{6.12}
\end{equation}
for some unitary $\alpha \colon {\mathcal D}_{R^{*}} \to {\mathcal 
D}_{T^{*}}$ and $\beta \colon{\mathcal D}_{R} \to {\mathcal D}_{T}$.
Indeed, if $T_i=URU^*$ for a unitary $U: \, \widetilde{\cX}\to \cX$,
then
\eqref{6.12} holds with $\alpha=U\vert_{{\mathcal D}_{T^*}}$ and 
$\beta=\bigoplus_1^d U\vert_{{\mathcal D}_{T}}$. Moreover,
it is easy to see that the unitary operators $X_{T} \colon
\cD(T)^{\perp} 
\to \cR(T)^{\perp}$ and $X_{R} \colon \cD(R)^{\perp} \to
\cR(R)^{\perp}$ 
are unitarily equivalent.  This suggests that we define an
equivalence relation on
admissible triples:  {\em we say that the two admissible triples 
$(S, {\mathbb K}, X)$ and $(S', {\mathbb K}', X')$ are {\em
equivalent} if 

\smallskip

(i) $(S, {\mathbb K})$ and $(S', {\mathbb K}')$ jointly coincide, and 

\smallskip

(ii) $X \colon \cD^{\perp} \to \cR^{\perp}$ and $X' \colon 
\cD^{\prime \perp} \to \cR^{\prime \perp}$ are unitarily equivalent.}

\smallskip

The discussion above shows that the equivalence class of
$(\theta_{T}, {\mathbb K}_{T}, X_{T})$ is a unitary invariant for any 
\textbf{c.c.}~commutative row contraction $T$.  The next result 
gives the converse.

\begin{theorem}  \label{T:invtriple}
    Suppose that $T$ and $R$ are two \textbf{c.c.}~commutative row 
    contractions such that the associated characteristic triples 
    $(\theta_{T}, {\mathbb K}_{T}, X_{T})$ and $(\theta_{R}, 
    {\mathbb K}_{R}, X_{R})$ are equivalent as admissible triples.
Then $T$ and $R$ are unitarily equivalent.  
\end{theorem}

\begin{proof}  We have seen that $T$ is unitarily equivalent to 
 $A(T)^{*}$ and $R$ is unitarily equivalent to $A(R)^{*}$ where
    $A(T)$ (respectively $A(R)$)
    appears in the \textbf{t.c.f.m.}~colligation  
    $\bU(T)$ (respectively $\bU(R)$) \eqref{bU1} and \eqref{bU2}
associated with 
    $(\theta_{T}, {\mathbb K}_{T}, X_{T})$ (respectively 
    $(\theta_{R}, {\mathbb K}_{R}, X_{R})$). By applying unitary 
    changes of bases on the input and output spaces
    coming from the joint coincidence of 
    $(\theta_{T}, {\mathbb K}_{T})$ and 
    $(\theta_{R}, {\mathbb K}_{R}$) in the input and output spaces, 
    we may even assume that $\theta_{T} = \theta_{R}$ and 
    ${\mathbb K}_{T} = {\mathbb K}_{R}$.  
    It then follows that $V_{T} = V_{R}$ ($V_{T}$ as in \eqref{bU2}), 
    and the assumption that $X_{T}$ is unitarily equivalent to $X_{R}$
    then implies that  $A(T)$ is unitarily equivalent to $A(R)$.  As $Y$ is unitarily 
    equivalent to $A(T)^{*}$ and $R$ is unitarily equivalent to 
    $A(R)^{*}$, it now follows that $T$  and $R$ are unitarily 
    equivalent to each other.
 \end{proof}
 
 \begin{remark} {\em  Note that if $(S, {\mathbb K}, X)$ is an 
     admissible triple with associated \textbf{t.c.f.m.} colligation
     $\left[ \begin{smallmatrix} A & B \\ C & D \end{smallmatrix} 
     \right]$, then $A^{*} =  \begin{bmatrix} A_{1}^{*} & \cdots & 
     A_{d}^{*} \end{bmatrix}$ is a \textbf{c.c.}~commutative row 
     contraction.  Moreover, if $(S', {\mathbb K}', X')$ is another 
     admissible triple, that $A^{*}$ and $A^{\prime *}$ are unitarily 
     equivalent if and only if the associated admissible triples $(S, 
     {\mathbb K}, X)$ and $(S', {\mathbb K}', X')$ are equivalent as 
     triples. Thus admissible triples can be viewed as providing a 
     parametrization of \textbf{c.c.}~commutative row contractions.
     This parametrization is somewhat crude, however, since there is 
     no explicit way (1) to write down all the Agler decompositions 
     ${\mathbb K}$ associated with $S$, and (2) pick out among them 
     which ${\mathbb K}$ lead to reproducing kernel spaces 
     $\cH({\mathbb K})$ so that (a) the dimension criterion $\dim 
     \cD^{\perp} = \dim \cR^{\perp}$ holds, and (b) there exist a 
     unitary $X \colon \cD^{\perp} \to \cR^{\perp}$ giving rise to a 
     commutative $A$ in the associated \textbf{t.c.f.m.} colligation.

     In the \textbf{c.n.c.}~case as developed in \cite{BES, BES2}, 
     the complete invariant for a \textbf{c.n.c.}~commutative row 
     contraction $T$ is just the characteristic function $\theta_{T}$
     and there is a version of the Sz.-Nagy-Foias functional model space. 
     Even in this case, there remains the issue of characterizing 
     which pure Schur-class functions $S$ coincide with the characteristic 
     function $\theta_{T}$ of a \textbf{c.n.c.}~commutative row .
     contraction;
     the following partial result on this issue  follows by combining 
     Theorem \ref{T:comreal} above with Proposition 6.1 in \cite{BBF2b}. 
     
     \begin{theorem}
If the function $S \in {\mathcal S}_{d}(\cU, \cY)$
coincides with a characteristic function  $\theta_T$ of a
commutative {\bf c.n.c.}~row contraction $T$, then 
$S$ is pure, the space $\cH(K_S)$  is ${\bf M}_z^*$-invariant,
inequality \eqref{3.1} holds, and finally,
\begin{equation}
\dim \, {\rm Ker} A^*\vert_{{\mathcal D}^\perp}=\dim \,\cU_S^0,
\label{1}
\end{equation}
where
$$
A=\begin{bmatrix}M_{z_1}^*\vert_{\cH(K_{S})} \\ \vdots \\
M_{z_d}^*\vert_{\cH(K_{S})}\end{bmatrix},\quad
\cU_S^0=\{u\in\cU: \, S(z)u\equiv 0 \}.
$$
\label{T:charfunc}
\end{theorem}

     In general the construction of Agler decompositions ${\mathbb 
     K}$ from a given Schur-class function $S$ is 
     mysterious.\footnote{However one 
     case where a multitude of Agler decompositions can be written 
     down for a given function $S = 0$ is presented as Example 
     \ref{E:spherical} below.} In particular, 
     we do not know any characterization of when the Agler 
     decomposition is unique.  If it were the case 
     that there is a unique Agler decomposition in case $S = 
     \theta_{T}$ with $T$ \textbf{c.n.c.}, then one could see Theorem 
     \ref{P:6.4} as a corollary to Theorem \ref{T:invtriple}.
     }\end{remark}
 
 In case  $\dim \cD(T)^{\perp} = \dim \cR(T)^{\perp} = 0$, 
 then the third object in the admissible triple $X_{T}$ is trivial
 and can be ignored.  To analyze this situation, let us introduce another 
 subspace associated with a \textbf{c.c.}~commutative row 
 contraction $T$, namely
 \begin{align}
     {\mathcal M}^{(2)}_T= &\left\{x\in\cX \colon 
     D_{T^*}(I_{\cX}-Z_{\cX}(z)T^*)^{-1}x\equiv 0 \text{ and } \right. \notag\\
     &\quad \left. 
     D_{T}(I_{\cX^{d}}-Z_{\cX}(z)^*T)^{-1}Z_{\cX}(z)^*x\equiv 0\right\}.\label{6.10''}
\end{align}
It is readily seen that ${\mathcal M}^{(1)}_T\subset {\mathcal 
M}^{(2)}_T$.
Indeed, if $D_{T}(I-Z(z)^*T)^{-1}{\mathcal I}_{i}^*x\equiv 0$ for 
$i=1,\ldots,d$, then also
\begin{align*}
0\equiv \sum_{i=1}^d z_iD_{T}(I_{\cX^{d}}-Z_{\cX}(z)^*T)^{-1}{\mathcal I}_{i}^*x
&=D_{T}(I_{\cX^{d}}-Z_{\cX}(z)^*T)^{-1}\sum_{i=1}^d z_i{\mathcal I}_{i}^*x\\
&= D_{T}(I_{\cX^{d}}-Z_{\cX}(z)^{*}T)^{-1} Z_{\cX}(z)^{*}x.
\end{align*}
In case $d=1$, we have ${\mathcal M}^{(1)}_{T} = 
{\mathcal M}^{(2)}_{T}$ and either
 space is the maximal
reducing space for $T$ on which $T$ is unitary.   Hence 
$\cM_{T}^{(1)} = \{0\}$ and $\cM_{T}^{(2)} = \{0\}$ are both 
equivalent to $T$ being \textbf{c.n.u.} for the single-variable $d=1$ 
case.  For the multivariable setting, as we have already taken 
$\cM_{T}^{(1)} = 0$ as the definition of $T$ being \textbf{c.c.},
we make the following definition.

\begin{definition}
Given a commutative row contraction $T$ we say that $T$ is {\em
strongly closely connected} (\textbf{strongly c.c.})~if 
${\mathcal M}^{(2)}_T=\{0\}$ with $\cM_{T}^{(2)}$ as in 
\eqref{6.10''}.  Equivalently,
\begin{equation}  \label{span-strongcc}
    \cX = \bigvee \left\{ \operatorname{Ran} {\mathbf T}^{n} 
    D_{T^{*}}, \, \operatorname{Ran} \left( \sum_{k=1}^{d} 
    {\mathcal I}_{k}^{*} X_{n-e_{k}} D_{T} \right) 
     \colon n \in {\mathbb Z}^{d}_{+}, 
     \right\}.
    \end{equation}
where $X_n$ is given in \eqref{Xn} and where $e_k$ stands for the element in ${\mathbb
Z}_+^d$ with one in the $k$-th slot and zeros in all other slots.
\label{D:6.5''}
\end{definition}

From the chain of containments \eqref{chain1} combined with the 
observation made above that $\cM_{T}^{(1)} \subset \cM_{T}^{(2)}$, we 
get
\begin{align}
&  \bigvee \left\{ \operatorname{Ran}{\mathbf T}^{n} D_{T^{*}} \colon n 
\in {\mathbb Z}^{d}_{+} \right\} \notag \\
& \quad \subset (\cM_{T}^{(2)})^{\perp} = \left\{ \operatorname{Ran}{\mathbf T}^{n} D_{T^{*}}, 
\operatorname{Ran} \left(  \sum_{k=1}^{d} {\mathcal I}_{k}^{*} 
X_{n-e_{k}} D_{T} \right), n \in {\mathbb Z}^{d}_{+} \right\} \notag \\
& \quad \subset (\cM^{(1)})^{\perp} = \bigvee \left\{ \operatorname{Ran}{\mathbf T}^{n} D_{T^{*}}, 
\operatorname{Ran}{\mathcal I}_{k}^{*} X_{n}D_{T} \colon n \in {\mathbb 
Z}^{d}_{+}, k=1, \dots, d \right\}  
\notag \\
& \quad \subset \bigvee \left\{ \operatorname{Ran}{\mathbf T}^{*n} D_{T^{*}}, 
\operatorname{Ran} {\mathbf T}^{n} {\mathbf T}^{*m} {\mathcal 
I}_{k}^{*} D_{T}\colon n, 
m \in {\mathbb Z}^{d}_{+},\, k=1, \dots, d  \right\} = \cX,
 \label{chain2}
\end{align}
from which we see that \textbf{c.n.c.}~$\Rightarrow$ \textbf{ 
strongly c.c.}~$\Rightarrow$ 
\textbf{c.c.}~$\Rightarrow$\textbf{c.n.u.} for a commutative row 
contraction ${\mathbf T}$ (where the last property 
\textbf{c.n.u.}~holds for {\em any} commutative row contraction.

\smallskip

One can check that $\cM_{T}^{(2)} = \{0\}$ amounts to the condition 
that $\cR(T)^{\perp} = \{0\}$ (here ${\mathcal R}(T)$ is as in 
\eqref{bU2}). Thus the characteristic triple 
$(\theta_{T}, {\mathbb K}_{T}, X_{T})$ for a \textbf{strongly 
c.c.}~commutative row contraction collapses to 
$(\theta_{T}, {\mathbb K}_{T}, 0\}$.  Given two \textbf{strongly 
c.c.}~commutative row contractions $T$ and $R$, equivalence 
of the characteristic triples $(\theta_{T}, {\mathbb K}_{T}, 0)$, 
$(\theta_{R}, {\mathbb K}_{R}, 0)$ collapses to joint coincidence of 
the characteristic function/Agler decomposition pairs $(\theta_{T}, 
{\mathbb K}_{T})$, $(\theta_{R}, {\mathbb K}_{R})$.  Thus the 
following result is an immediate special case of Theorem 
\ref{T:invtriple}. 

\begin{theorem}  \label{T:invpair}
Let $T$ and $R$ be two \textbf{strongly c.c.}~commutative row  contractions 
and let us assume that the associated 
characteristic function/Agler decomposition pairs $(\theta_T, 
{\mathbb K}_{T})$ and 
$(\theta_R, {\mathbb K}_{R})$ jointly coincide (i.e., 
let us assume that \eqref{6.12} holds).
Then $T$ and $R$ are  unitarily equivalent.
\end{theorem}

\begin{remark}  {\em In Theorem \ref{T:invpair} it is enough to 
    assume that $T$ is \textbf{strongly c.c.}~ with $R$ only 
    \textbf{c.c.}~or vice versa.
} \end{remark}

We have seen that the model-theory results are the best for the case
where the commutative row contraction is
\textbf{c.n.c.}  It essentially follows from the definitions that
 any commutative row-contractive $d$-tuple ${\mathbf T} = (T_{1}, \dots,
T_{d})$ can be decomposed as
$$
{\mathbf T} = \begin{bmatrix} {\mathbf T}_{cnc} & {\boldsymbol \Gamma} \\ 0 & {\mathbf
T}_{c} \end{bmatrix} =
\left( \begin{bmatrix} T_{cnc,1} & \Gamma_{1}  \\ 0 & T_{c}
\end{bmatrix}, \dots, \begin{bmatrix} T_{cnc,d} & \Gamma_{d} \\ 0 &
T_{c,d} \end{bmatrix} \right)
$$
where ${\mathbf T}_{cnu}$ is \textbf{c.n.c.} while ${\mathbf T}_{c}$
is \textbf{coisometric}, i.e., the operator block-row matrix
$T_c=\begin{bmatrix} T_{c,1} & \cdots & T_{c,d} \end{bmatrix}$ is
coisometric as an operator from $\cX_{c}^{d}$ to $\cX_{d}$:
$$
 T_{c,1} T_{c,1}^{*} + \cdots + T_{c,d} T_{c,d}^{*} = I_{\cX_{c}}.
$$
It is known (see \cite{Ath2}) that any column isometry such as
$T_{c}^{*}$ (sometimes also called a {\em spherical isometry})
is jointly subnormal and extends to a spherical unitary, i.e., a
commutative $d$-tuple ${\mathbf N} = (N_{1}, \dots, N_{d})$ with
joint spectral measure supported on the unit sphere 
$\partial {\mathbb B}^{d}$. Thus there is rather complete
unitary-equivalence classification theory (in terms of the
absolutely-continuous equivalence class of a spectral measure
supported on $\partial {\mathbb B}^{d}$ together with specification
of a multiplicity function) for spherical-unitary $d$-tuples.
Nevertheless it makes sense to apply our \textbf{t.c.f.m.}-model
theory to spherical-unitary tuples.   In this case the
characteristic function $\theta_{T}$ has values in $\cL(\cX^{d},
\{0\})$ and is thus trivial.  Thus this case separates out the extra
invariant (i.e., the Agler-decomposition kernel ${\mathbb K}_{T}$) as
the only object of interest.  Since the space $\cY$ is trivial in
this case, the two-component Agler-decomposition
then collapses to the single-component Agler decomposition occurring for the weakly 
isometric case as sketched in Section \ref{T}.  To make all objects explicitly
computable, in the following example, we specialize even further to
the simplest case where $\bT$ is just a pair of complex numbers
$(\lam_{1}, \lam_{2})$ on the boundary of the unit ball in ${\mathbb
C}^{2}$.

\begin{example}  \label{E:spherical} {\em
Let 
\begin{equation}
T=\begin{bmatrix}\lam_{1} & \lam_{2}\end{bmatrix}, \quad\mbox{where}\quad
\lam_1,\lam_2\in\C\quad\mbox{and}\quad |\lam_{1}|^{2} +
|\lam_{2}|^{2} = 1.
\label{t-lam}
\end{equation}
Thus, $\lambda=(\lam_{1}, \lam_{2})$ is fixed point on the boundary
of the unit ball ${\mathbb B}^{2}$ in ${\mathbb C}^{2}$ and we  
view $T$ as a commutative row-contraction on the Hilbert space $\cX = {\mathbb C}$,
to which our model theory applies.  Our goal is to compute explicitly the model
characteristic-function/Agler-decomposition pair $(\theta_{T},{\mathbb K}_{T})$ for this case.

\smallskip

Since $TT^*=1$, it follows that $D_{T^{*}} = 0$ (as an operator on ${\mathbb
C}$) and that $D_{T}$ is the orthogonal projection onto the orthogonal complement of 
${\rm Ran}\, T^*$. Thus, ${\mathcal D}_{T}$ is spanned by the vector 
$\left[\begin{smallmatrix}
-\lam_{2} \\ \lam_{1} \end{smallmatrix}\right]$, and if we write the
characteristic colligation ${\mathbf U}_{T}=\left[ \begin{smallmatrix} T^{*} & D_{T}
\\ D_{T^{*}} & -T \end{smallmatrix} \right] \colon \left[
\begin{smallmatrix} {\mathbb C} \\ {\mathcal D}_{T} \end{smallmatrix} \right]
    \to \left[ \begin{smallmatrix} {\mathbb C}^{2} \\ \{0\} \end{smallmatrix}
    \right]$ as a matrix with respect to the choice of basis
$\left[\begin{smallmatrix}-\lam_{2} \\ \lam_{1} \end{smallmatrix}\right]$
 for ${\mathcal D}_{T}$, we arrive at
$$
{\mathbf U}_{T} = \begin{bmatrix} A & B \\ C & D \end{bmatrix} =
\left[ \begin{array}{c|c} \overline{\lam}_{1} & -\lam_{2}\\ 
\overline{\lam}_{2} & \lam_{1} \\ \hline  \\
0 & 0  \end{array}  \right] \colon \begin{bmatrix} {\mathbb
C} \\  {\mathbb C}  \end{bmatrix} \to \begin{bmatrix} {\mathbb C}^{2}
\\ \{0\} \end{bmatrix}.
$$
 Then the characteristic function $\theta_{T}(z)
= D + C(I - Z_{\rm row}(z) A)^{-1} Z_{\rm row}(z) B$
is trivial since it lands in the zero-dimensional space.  The
Agler-decomposition kernel ${\mathbb K}_{T}(z, \zeta)$ given by
\eqref{6.11} collapses to the lower diagonal block $\Phi(z, \zeta) =
\left[ \Phi_{ij}(z, \zeta) \right]$ ($i,j=1,2$), again since
$\cY = {\mathcal D}_{T^{*}} = \{0\}$.  We then compute
\begin{align*}
B^{*} (I - Z_{\rm row}(z)^{*} A^{*})^{-1} & =
\begin{bmatrix} -\overline{\lam}_{2} & \overline{\lam}_{1}
    \end{bmatrix}\begin{bmatrix} 1 - \overline{z}_{1} \lam_{1} & - \overline{z}_{1}
     \lam_{2} \\ - \overline{z}_{2} \lam_{1} & 1 - \overline{z}_{2}
     \lam_{2} \end{bmatrix} ^{-1}\\
& =\frac{1}{d(z, \lam)}
    \begin{bmatrix} -\overline{\lam}_{2} & \overline{\lam}_{1}
    \end{bmatrix}\begin{bmatrix} 1 -
    \overline{z}_{2} \lam_{2} & \overline{z}_{1} \lam_{2} \\
    \overline{z}_{2} \lam_{1} & 1 - \overline{z}_{1} \lam_{1}
\end{bmatrix} \\
& = \frac{1}{d(z, \lam)} \begin{bmatrix} \overline{z}_{2} -
\overline{\lam}_{2} & -(\overline{z}_{1} - \overline{\lam}_{1}),
\end{bmatrix}
\end{align*}
where we made use of the assumed identity $|\lam_{1}|^{2} +
|\lam_{2}|^{2} = 1$ and where we have set
$$
d(z, \lam) = \det (I - Z_{\rm row}(z)^*A^*)=1-A^*Z_{\rm row}(z)^* =
1 -\overline{z}_{1} \lam_{1} - \overline{z}_{2} \lam_{2}.
$$
Thus ${\mathbb G}_{\lam}(z)$ in \eqref{6.11} (with $\lam$ in place of $T$ and with respect 
to our choice of basis for ${\mathcal D}_T$) becomes
$$
{\mathbb G}_{T}(z) = \frac{1}{d(z, \lam)} \begin{bmatrix}
\overline{z}_{2} - \overline{\lam}_{2} \\ - (\overline{z}_{1} -
\overline{\lam}_{1}) \end{bmatrix}
$$
and hence
\begin{align}
    {\mathbb K}_{\lam}(z, \zeta) & = {\mathbb G}_{\lam}(z)
    {\mathbb G}_{\lam}(\zeta)^{*}  \notag \\
    & = \frac{1}{d(z, \lam) \overline{d(\zeta, \lam)}}
    \begin{bmatrix} \overline{z}_{2} - \overline{\lam}_{2} \\
        -(\overline{z}_{1} - \overline{\lam}_{1}) \end{bmatrix}
        \begin{bmatrix} \zeta_{2} - \lam_{2} & -(\zeta_{1} -
            \lam_{1}) \end{bmatrix}.  \label{Klam}
\end{align}
The expected Agler decomposition
$$
 I - \theta_{T}(z)^*\theta_{T}(\zeta)= \sum_{k=1}^{2} \Phi_{kk}(z, \zeta) -
 \sum_{i,j=1}^{2} \overline{z}_{i} \zeta_{j} \Phi_{ij}(z, \zeta)
$$
can be expressed as
\begin{align*}
& \left(1 - \overline{z}_{1} \lam_{1} - \overline{z}_{2}
\lam_{2}\right)
  \left( 1 - \zeta_{1} \overline{\lam}_{1} - \zeta_{2}
 \overline{\lam}_{2} \right)  =   \\
& \qquad  (\overline{z}_{2} - \overline{\lam}_{2}) (\zeta_{2} - \lam_{2})
  + (\overline{z}_{1} - \overline{\lam}_{1}) ( \zeta_{1} - \lam_{1}) \\
 &  \qquad
 - \overline{z}_{1} \zeta_{1}(\overline{z}_{2} - \overline{\lam}_{2})
      (\zeta_{2} - \lam_{2}) + \overline{z}_{1} \zeta_{2}(\overline{z}_{2}
       - \overline{\lam}_{2}) ( \zeta_{1} - \lam_{1}) \\
 & \qquad
 + \overline{z}_{2} \zeta_{1} (\overline{z}_{1} -
 \overline{\lam}_{1}) (\zeta_{2} - \lam_{2})
 - \overline{z}_{2} \zeta_{2} (\overline{z}_{1} - \overline{\lam}_{1})
 (\zeta_{1} - \lam_{1}).
\end{align*}
This identity in turn can be checked directly by a routine but
tedious calculation (or as an exercise for a software package such as
{\tt MATHEMATICA}). For this simple example it is easily checked that we are in 
the \textbf{strongly c.c.}~case.  Thus by Theorem \ref{T:invpair}, the
characteristic function/Agler decomposition pair $(0, {\mathbb
K}_{\lam})$ is a complete unitary invariant for $\lam=\begin{bmatrix}\lam_1 & 
\lam_2\end{bmatrix}$ within the class of commutative \textbf{strongly c.c.}~row-
contractive operator 2-tuples.  Since the matrix entries of ${\mathbb K}_{\lam}$ are
scalar, it is clear that two such kernels ${\mathbb K}_{\lam}$ and
${\mathbb K}_{\lam'}$ coincide if and only if they are identical.  It
is also elementary that two such $\lam$'s are unitarily equivalent as
operator tuples if and only if they are identical.  We conclude
as a consequence of Theorem \ref{T:invpair} that,
given two points $\lam$ and $\lam'$ on $\partial {\mathbb B}^{2}$,
then
${\mathbb K}_{\lam} = {\mathbb K}_{\lam'}$ if and only if $\lam =
\lam'$---a point which of course can also be verified directly from
the formula \eqref{Klam}.  In summary, for this case we have used a more
complicated object ${\mathbb K}_{\lam}$ to classify a much simpler
object $\lam = (\lam_{1}, \lam_{2})$;  presumably there are other
examples $\bT = (T_{1}, \dots, T_{d})$ where the characteristic pair
$(\theta_{T}, {\mathbb K}_{T})$ is a simpler object than $T$
and for which the \textbf{t.c.f.m.} associated with $(\theta_{T},
{\mathbb K}_{T})$ sheds some light on the structure of $T$.

\smallskip

Note that in this example we have arrived at a whole family ${\mathbb
K}_{\lam}(z, \zeta)$ of essentially different Agler decompositions for the fixed Schur-class 
function $S(z) =  0 \colon {\mathbb C} \to \{0\}$, indexed by a point $\lam
\in \partial {\mathbb B}^{2}$.  In general, identification of a
family of row-contractive operator tuples $\bT_{\lam}$ all having the
same characteristic function $\theta_{T_{\lam}} = S$ leads to the
construction of a whole family $\{{\mathbb K}_{\bT_{\lam}}\}$ of
Agler decompositions for the fixed Schur-class function $S$.  This
illustrates the non-uniqueness of Agler decompositions for a given
$S$ and may lead to other examples where a whole family of distinct
Agler decompositions can be exhibited explicitly.}
\end{example}

\section{Noncommutative Agler decompositions}  \label{S:NAD}
\setcounter{equation}{0}
     
     It is possible also to study noncommutative Agler decompositions 
     for the noncommutative characteristic function $\theta_{T, 
     nc}(\bz)$ of a (possibly noncommutative) row contraction $T$ 
     as follows.  We first need to review some basic facts concerning 
     noncommutative kernels; a systematic treatment can be found in 
     \cite{BV-formal}.   

\smallskip
     
     A noncommutative kernel (with operator coefficients) is a formal 
     power series in two sets of noncommuting indeterminates $\bz = 
     (\bz_{1}, \dots, \bz_{d})$ and $\bzeta = (\bzeta_{1}, \dots, 
     \bzeta_{d})$ of the form
     $$
     {\mathbb K}(\bz, \bzeta) = \sum_{\alpha, \beta \in \cF_{d}} 
     {\mathbb K}_{\alpha, \beta} \bz^{\alpha} \bzeta^{\beta}.
     $$
     While we assume that the $\bz_{1}, \dots, \bz_{d}$ do not commute 
     with each other and similarly for $\bzeta_{1}, \dots, 
     \bzeta_{d}$, it is convenient to assume that the $\bz$'s commute 
     with the $\bzeta$'s:  $\bz_{i} \bzeta_{j} = \bzeta_{j} \bz_{i}$ 
     for all $i,j = 1, \dots, d$.  Let us say that the noncommutative 
     kernel ${\mathbb K}(\bz, \bzeta)$ is {\em positive} if it has a 
     Kolmogorov decomposition of the form
     $$
     {\mathbb K}(\bz, \bzeta)  = H(\bz) H(\bzeta)^{*}.
     $$
     Here $H(\bzeta)^{*} = \sum_{v \in \cF_{d}} H_{v}^{*} 
     \bzeta^{v^{\top}}$ if $H(\bz) = \sum_{v \in \cF_{d}} H_{v} 
     \bz^{v}$.  Note that here we follow the conventions of 
     \cite{BV-formal, Cuntz-scat} and avoid introduction of formal 
     conjugate variables $\overline{\bzeta}_{1}, \dots 
     \overline{\bzeta}_{d}$:  we define $(\bzeta^{v})^{*} = 
     \bzeta^{v^{\top}}$ where $v^{\top} = i_{1} \dots i_{N}$ is the 
     transpose of $v = i_{N} \dots i_{1} \in \cF_{d}$. If the formal 
     power series in noncommuting indeterminates $\bz = (\bz_{1}, 
     \dots, \bz_{d})$ has a realization of the form
     $$
     S(\bz) = D + C (I - Z_{\cX}(\bz) A)^{-1}  Z_{\cX}(\bz) B,
     $$
     with colligation matrix  ${\mathbf U} = \left[ \begin{smallmatrix} A & B \\ C & D 
 \end{smallmatrix} \right] = \left[ \begin{smallmatrix} A_{1} & 
 B_{1}  \\ \vdots & \vdots \\ A_{d} & B_{d} \\ C & D 
\end{smallmatrix} \right] \colon  \left[ \begin{smallmatrix} \cX  \\ 
\cU \end{smallmatrix} \right]  \to  \left[ \begin{smallmatrix} 
\cX^{d} \\ \cY \end{smallmatrix} \right]$,  then the noncommutative 
kernel
\begin{align}
   &  {\mathbb K}(\bz, \bzeta)  = \label{NCbigkerreal}\\
    & \quad  \begin{bmatrix} C (I_{\cX} - 
    Z_{\cX}(\bz) A)^{-1} \\ B^{*} (I_{\cX^{d}} - Z_{\cX}(\bz)^{*} 
    A^{*})^{-1} \end{bmatrix}
    \begin{bmatrix} (I_{\cX} - A^{*} Z_{\cX}(\bzeta)^{*})^{-1} C^{*} &
	(I_{\cX^{d}} - A Z_{\cX}(\bzeta) )^{-1} B \end{bmatrix}\notag
\end{align}
is positive and has a decomposition of the form 
\begin{equation}   \label{NCbigker}
    {\mathbb K}(\bz, \bzeta) = \begin{bmatrix} K_{S}(\bz, \bzeta) & 
    \Psi_{1}(\bz, \bzeta) & \cdots & \Psi_{d}(\bz, \bzeta)  \\
    \Psi_{1}(\bzeta, \bz)^{*} & \Phi_{11}(\bz, \bzeta) & \cdots & \
    \Phi_{1d}(\bz, \bzeta)  \\
     \vdots & \vdots & & \vdots \\
    \Psi_{d}(\bzeta, \bz)^{*} & \Phi_{d1}(\bz, \bzeta) & \cdots & 
    \Phi_{dd}(\bz, \bzeta) \end{bmatrix}
    \end{equation}
    with 
    \begin{align}
	K_{S}(\bz, \bzeta)& = C (I - Z_{\cX}(\bz)A)^{-1} (I - A^{*} 
	Z_{\cX}(\bzeta)^{*})^{-1} C^{*},  \notag \\
	\Psi_{k}(\bz, \bzeta) & = C (I - Z_{\cX}(\bz)A)^{-1} 
	{\mathcal I}_{k}^{*} (I - A Z_{\cX}(\bzeta))^{-1} B, \notag  \\
	\Phi_{ij}(\bz, \bzeta) & = B^{*} (I - Z_{\cX}(\bz)^{*} 
	A^{*})^{-1} {\mathcal I}_{i} {\mathcal I}_{j}^{*}
	(I - A Z(\bzeta))^{-1} B.
	\label{NCentries}
\end{align}
Furthermore, by making use of the assumed unitary property of the 
colligation matrix ${\mathbf U}$, one can check that the block matrix 
entries of ${\mathbb K}$ in \eqref{NCbigker} satisfy the additional 
identities
\begin{align}
    K_{S}(\bz, \bzeta)&  = k_{Sz}(\bz, \bzeta) I_{\cY} - S(\bz) \left( 
    k_{Sz}(\bz, \bzeta) I_{\cU} \right) S(\bzeta)^{*}, \notag \\
    S(\bz) - S(\bzeta) & = \sum_{k=1}^{d} \left[ \Psi_{k}(\bz, \bzeta) 
    \bz_{k} - \bzeta_{k} \Psi_{k}(\bz, \bzeta) \right], \notag \\
    I - S(\bz)^{*} S(\bzeta) & = \sum_{k=1}^{d} \Phi_{kk}(\bz, 
    \bzeta)  -  \sum_{i,j=1}^{d} \bzeta_{j} \Phi_{ij}(\bz, \bzeta) 
    \bz_{i}.  \label{NCAglerdecom}
\end{align}
Here we use the {\em noncommutative formal Szeg\H{o} kernel} 
$k_{Sz}(\bz, \bzeta)$ given by
$$
  k_{Sz}(\bz, \bzeta) = \sum_{\alpha \in \cF_{d}} \bz^{\alpha} 
  \bzeta^{\alpha^{\top}}.
$$

Conversely, given a formal power series $S(\bz) = \sum_{\alpha \in 
\cF_{d}} S_{\alpha} \bz^{\alpha}$ with coefficients $S_{\alpha} \in 
\cL(\cU, \cY)$, we say that a kernel ${\mathbb K}$ of the form
\eqref{NCbigker} is an {\em Agler decomposition for} $S$ if the 
relations \eqref{NCAglerdecom} all hold true.
In case the colligation matrix ${\mathbf U}$ has the Halmos-dilation 
form ${\mathbf U} =  \left[ \begin{smallmatrix} T^{*} & D_{T} \\ 
D_{T^{*}} & - T \end{smallmatrix} \right] \colon \left[ 
\begin{smallmatrix} \cX \\ {\mathcal D}_{T} \end{smallmatrix} \right] 
    \to \left[ \begin{smallmatrix} \cX^{d} \\ {\mathcal D}_{T^{*}} 
\end{smallmatrix} \right]$ for a row-contractive operator 
$T =\begin{bmatrix}T_{1}, \dots, T_{d}\end{bmatrix}$ from $\cX^{d}$ to $\cX$, then the pair $(S = 
\theta_{T, nc}, {\mathbb K}_{T. nc})$ (where ${\mathbb K}_{T, nc}$ 
is  as in \eqref{NCbigkerreal}) is a unitary invariant for $T$.  
Several questions arise:  (1) to what 
extent is $(\theta_{T, nc}, {\mathbb K}_{T, nc})$ a complete unitary 
invariant for $T$, and (2) to what extent can we use a 
given noncommutative Schur-class multiplier/noncommutative 
Agler-decomposition pair $(S, {\mathbb K})$ to construct a 
row contraction $T$ such that $(S, {\mathbb K}) 
= (\theta_{T, nc}, {\mathbb K}_{T, nc})$?

\smallskip

The first question is more elementary than the second and can be 
resolved as follows.  From the formula \eqref{NCbigkerreal} we see 
that ${\mathbb K}_{T, nc}(\bz, \bzeta)$ is given by
\begin{align*}
    & {\mathbb K}_{T, nc}(\bz,\bzeta) = \\
    & \quad 
    \begin{bmatrix} D_{T^{*}} (I_{\cX} - Z_{\cX}(\bz) T^{*})^{-1} \\
	D_{T} (I_{\cX^{d}} - Z_{\cX}(\bz) T)^{-1} \end{bmatrix}
\begin{bmatrix} I_{\cX} - T Z_{\cX}(\bzeta)^{*})^{-1} D_{T^{*}} &
    (I _{\cX^{d}} - T^{*} Z_{\cX}(\bzeta))^{-1} D_{T} \end{bmatrix}.
    \end{align*}
 From the formula \eqref{NCbigker} and the noncommutative Agler 
 decomposition formulas \eqref{NCAglerdecom}, we see that the upper 
 diagonal bock and the off-diagonal blocks are already uniquely 
 determined by $S(\bz) = \theta_{T, nc}(\bz)$.  The lower diagonal 
 block has the form  $\Phi(\bz, \bzeta) = \left[ \Phi_{ij}(\bz, 
 \bzeta) \right]_{i,j=1}^{d}$ where
 $$
 \Phi_{ij}(\bz, \bzeta) = D_{T} (I_{\cX^{d}} - Z_{\cX}(\bz)^{*} 
 T)^{-1} {\mathcal I}_{i} {\mathcal I}_{j}^{*} (I_{\cX^{d}} - T^{*} 
 Z_{\cX}(\bzeta))^{-1} D_{T}.
 $$
 It then follows that
 \begin{align*}
     & \sum_{i,j=1}^{d} \bzeta_{j} \Phi_{ij}(\bz, \bzeta) \bz_{i} \\
     & \quad = D_{T }(I_{\cX^{d}} - Z_{\cX}(\bz)^{*} T)^{-1} 
     Z_{\cX}(\bz)^{*} Z_{\cX}(\bzeta) (I_{\cX^{d}} - 
     Z_{\cX}(\bzeta))^{-1} D_{T}  \\
     & \quad 
     = D_{T} Z_{\cX}(\bz)^{*} \left( I_{\cX} - T Z_{\cX}(\bz)^{*}\right)^{-1} 
     \left( I_{\cX} - Z_{\cX}(\bzeta) T^{*} \right)^{-1} 
     Z_{\cX}(\bzeta) D_{T}  \\
     & \quad = 
     \sum_{i,j=1}^{d} \bz_{i} D_{T} {\mathcal I}_{i} \left(I_{\cX} - T 
     Z_{\cX}(\bz) \right)^{*} \left( I_{\cX} - Z_{\cX}(\bzeta) T^{*} 
     \right)^{-1} {\mathcal I}_{j}^{*} D_{T} \bzeta_{j}.
 \end{align*}
 In the present noncommutative setting,  any collection of nonzero formal 
 power series of the form
 $$
 \{ \bz_{i}G_{ij}(\bz, \bzeta) \bzeta_{j} \colon i,j=1, \dots, d\}
 $$
 is linearly independent and it follows that 
 knowledge of $\Phi_{ij}(\bz, \bzeta)$ uniquely determines the 
 modified kernels
 $$
 \widetilde \Phi_{ij}(\bz, \bzeta): = D_{T} {\mathcal I}_{i} 
 \left(I_{\cX} - T Z_{\cX}(\bz)^{*} \right)^{-1}
 \left( I_{\cX} - Z_{\cX}(\bzeta) T^{*} \right)^{-1} {\mathcal 
 I}_{j}^{*} D_{T}, \quad i,j = 1, \dots, d
 $$
 as well.  Using the formal power series expansion
 $$
 \left(I_{\cX} - T Z_{\cX}(\bz)^{*}\right)^{-1} = \sum_{\alpha \in 
 \cF_{d}}  \bT^{\alpha} \bz^{\alpha},
 $$
 by looking at the coefficient of $\bz^{\alpha} \bzeta^{\beta}$ in 
 the expansion for $\widetilde \Phi_{ij}(\bz, \bzeta)$ we see that 
 the $\widetilde \Phi_{ij}$'s determine uniquely the moments
 $$
 D_{T} {\mathcal I}_{i} \bT^{\alpha} \bT^{* \beta} {\mathcal 
 I}_{j}^{*} D_{T}, \quad \alpha, \beta \in \cF_{d} \; \text{ and } \; 
i,j =  1, \dots, d.
 $$
 Combining these moments with the moments 
 $$
  -T,\, D_{T^{*}} \bT^{*\alpha} {\mathcal I}_{j}^{*} D_{T} \colon 
  {\mathcal D}_{T} \to {\mathcal D}_{T^{*}}, \quad \alpha \in \cF_{d} 
  \; \text{ and } \; j = 1, \dots, d
  $$
  determined by the characteristic function $\theta_{T, nc}$ gives 
  us the list \eqref{exp-moments}.  By the result from 
  \cite{Cuntz-scat} we conclude that $(\theta_{T, nc}, {\mathbb 
  K}_{T, nc})$ is a complete unitary invariant for the general 
  \textbf{c.n.u.}~row contraction $T$ (in  particular, for commutative such $T$).  
	
 We conclude that the two-component Agler-decomposition approach to operator-model theory 
 (i.e., using the two-component Agler decomposition in addition to the 
 characteristic function as a unitary invariant) has mixed results.  
 In the commutative case, some additional information is added and 
 the characteristic-function/Ag\-ler-decomposition pair is definitive 
 in some special cases which go beyond the \textbf{c.n.c.}~case for 
 which the characteristic function $\theta_{T}$ alone is 
 definitive.  On the other hand, for the noncommutative setting, the 
 results from \cite{Cuntz-scat} can be reinterpreted to say that the 
 characteristic-function/Ag\-ler-decomposition pair is a complete 
 unitary invariant for the general \textbf{c.n.u.}~row contraction 
 $T$.  The following table summarizes our results on complete unitary invariants 
   for various classes of Hilbert-space row-contractive operator 
   $d$-tuples (note that the class on each line is a subclass of 
   the class on the next line):
$$
\left[ \begin{array}{c|c} \hline \\
    \textbf{operator $d$-tuple class} & \textbf{complete unitary 
    invariant}  \\
    \hline
   {\mathbf T} = \text{ commutative \textbf{c.n.c.}} & 
   \theta_{T}  \\
   {\mathbf T} = \text{ commutative \textbf{strongly c.c.}} & 
   (\theta_{T}, {\mathbb K}_{T}) \\
   {\mathbf T} = \text{ commutative \textbf{c.c.}} & 
   (\theta_T, {\mathbb K}_{T}, X_{T})    \\
   {\mathbf T} = \textbf{c.n.u.} & (\theta_{T, nc}, 
   {\mathbb K}_{T, nc}) \\
   \hline
   \end{array}  \right]
 $$

 As for the second question posed above (construction of a canonical 
 model for a given noncommutative Schur-class 
 function/Agler decomposition pair $(S, {\mathbb K}_{S})$), we can say the following.  
 By using the analysis in 
 \cite{Cuntz-scat}, from such a pair $(S, {\mathbb K}_{S})$ one can construct 
 a characteristic pair $(S, L)$ in the 
 sense of \cite{Cuntz-scat} from which one can construct a 
 noncommutative Sz.-Nagy-Foias functional-model space on which there 
 is a canonical choice of \textbf{c.n.u.}~row-contractive operator 
 $d$-tuple $\bT = \bT(S, L)$.  It should also be possible to 
 construct a noncommutative de Branges-Rovnyak model space directly 
 from the noncommutative Schur-class function/Agler decomposition 
 pair $(S, {\mathbb K}_{S})$; some machinery in this direction has 
 already been developed in \cite{BBF3}, but we leave the fleshing out 
 of the complete details for another occasion.

\end{document}